\documentclass[a4paper,10pt]{siamart190516} 
\usepackage[utf8]{inputenc}
\usepackage[english]{babel}

 \usepackage[sort]{cite}

\usepackage{amsmath,amssymb,amsfonts}
\usepackage{mathrsfs}
\usepackage{textcomp}
\usepackage{graphicx} 
\usepackage{xcolor}
\usepackage[all,2cell]{xy}
\usepackage{tikz}
\usetikzlibrary{backgrounds,patterns,arrows,positioning,fit,shapes.geometric,decorations.pathmorphing}

\usepackage{cleveref}
\Crefname{equation}{}{}
\crefname{equation}{}{}
\crefname{equation}{}{}
\crefname{figure}{Figure}{Figure}
\crefname{section}{Section}{Section}
\crefname{lemma}{Lemma}{Lemma}
\crefname{proposition}{Proposition}{Proposition}
\crefname{theorem}{Theorem}{Theorem}
\crefname{corollary}{Corollary}{Corollarie}
\crefname{definition}{Definition}{Definition}
\crefname{notation}{Notations}{Notation}
\crefname{remark}{Remark}{Remark}
\crefname{claim}{Claim}{Claim}  
\crefname{assumption}{Assumption}{Assumption}

\newtheorem{assumption}{Assumption}
\newtheorem{remark}[theorem]{Remark}

\newcommand{\changed}[1]{#1}

\newcommand{\Var}[1]{\mathcal{#1}}
\newcommand{\Pj}{\mathbb{P}}
\newcommand{\vect}[1]{#1}

\newcommand{\Tang}[2]{\mathrm{T}_{#1} {#2}}
\newcommand{\NSp}[2]{\mathrm{N}_{#1} {#2}}

\newcommand{\R}{\mathbb{R}}
\newcommand{\deriv}[2]{\mathrm{D}_{#2}#1}

\DeclareMathOperator*{\argmin}{argmin}

\newcommand{\dist}{\mathrm{dist}}
\newcommand{\SFF}{\mathit{I\!I}}

\title{The condition number of\\ Riemannian approximation problems\thanks{Submitted to the editors. \funding{P.~Breiding has received funding from the European Research Council (ERC) under the European Union's Horizon 2020 research and innovation programme (grant agreement No 787840). N.~Vannieuwenhoven was supported by a Postdoctoral Fellowship of the Research Foundation---Flanders (FWO) with project 12E8119N.}}}
\author{Paul Breiding\thanks{Institute of Mathematics, TU Berlin, Berlin, Germany. (\email{breiding@math.tu-berlin.de})} \and Nick Vannieuwenhoven\thanks{KU Leuven, Department of Computer Science, Leuven, Belgium. (\email{nick.vannieuwenhoven@kuleuven.be})}}
\allowdisplaybreaks

\begin{document}

\maketitle

\begin{abstract} 
We consider the local sensitivity of least-squares formulations of inverse problems. The sets of inputs and outputs of these problems are assumed to have the structures of Riemannian manifolds.
The problems we consider include the approximation problem of finding the nearest point on a Riemannian embedded submanifold from a given point in the ambient space.
We characterize the first-order sensitivity, i.e., condition number, of local minimizers and critical points to arbitrary perturbations of the input of the least-squares problem.
This condition number involves the Weingarten map of the input manifold, which measures the amount by which the input manifold curves in its ambient space.
We validate our main results through experiments with the $n$-camera triangulation  problem in computer vision.
\end{abstract}



\section{Introduction} \label{sec_intro}
A prototypical instance of the Riemannian optimization problems we consider here \changed{arises from} the following {inverse problem} (IP): Given an input $x\in \Var I$, compute $y\in \Var O$ such that $F(y)=x$. Herein, we are given a \textit{system} modeled as a smooth map~$F$ between a Riemannian manifold $\Var{O}$, the \textit{output manifold}, and a Riemannian embedded submanifold $\Var{I} \subset \R^n$, the \textit{input manifold}. \changed{In practice, however, we are often given a point $a\in\R^n$ in the \textit{ambient space} rather than $x\in\Var{I}$, for example because $a$ was obtained from noisy or inaccurate physical measurements. This leads to the Riemannian least squares optimization problem associated to the foregoing IP:}
\begin{align} \tag{PIP} \label{eqn_PIP}
 \argmin_{y \in \Var{O}} \frac{1}{2} \| F(y) - a \|^2.
\end{align}
In \cite{BV2017b} we called this type of problem a  \textit{parameter identification problem} (PIP).
Application areas where PIPs originate are \changed{data completion with low-rank matrix and tensor decompositions \cite{BV2018,Vandereycken2013,KSV2014,Steinlechner2016,dSH2015,HS2018}, geometric modeling \cite{CM2011, LW2008, SJ2008}, computer vision \cite{FL2001, HZ2003, Maybank1993}, and phase retrieval problems in signal processing \cite{Boumal2016,BEB2018}.}

The manifold $\Var{O}$ is called the output manifold because its elements are the outputs of the optimization problem. The outputs are the quantities of interest; for instance, they could be the unknown system parameters that one wants to identify based on a given observation $a$ and the forward description of the system $F$.
\changed{In order to distinguish between the inputs $x\in\Var{I}$ of the IP and the inputs $a\in\R^n$ of the optimization formulation \cref{eqn_PIP}, we refer to the latter as \textit{ambient inputs}.}
In this paper, we study the sensitivity of the output of the optimization problem \cref{eqn_PIP} \changed{(and an implicit generalization thereof)} when the ambient input $a\in\R^n$ is perturbed by considering its \textit{condition number} $\kappa$ at global and local minimizers and critical points. 
% In fact, we will consider a slightly more general version of \cref{eqn_PIP}, where an input $x\in \Var I$ may have several outputs $y\in \Var O$.

\subsection{The condition number} \label{sec_condition}
Condition numbers are one of the cornerstones in numerical analysis\cite{Wilkinson1963,Wilkinson1965,ANLA,matrix_computations,BC2013,Higham1996,BT1997}.
\changed{Its definition was introduced for} Riemannian manifolds already in the early days of numerical analysis. In 1966, Rice \cite{Rice1966} presented the definition in the case where input and output space are general nonlinear \emph{metric spaces}; the special case of Riemannian manifolds is \cite[Theorem 3]{Rice1966}.
\changed{Let $\Var{I}$ and $\Var{O}$ be Riemannian manifolds and assume that the map
% $F:\Var O \to \Var I$ is invertible with inverse 
$f:\Var I\to \Var O$ models a computational problem}. Then, Rice's definition of \changed{the condition number of $f$} at a point $x\in \Var{I}$ is
\begin{equation}\label{computational_problem2}
\kappa[f](x) := 
\lim_{\epsilon \to 0}\; \sup_{\substack{x' \in \Var{I},\\ \dist_{\Var{I}}(x,x') \leq \epsilon}}\; \frac{\dist_{\Var{O}}(f(x),f(x'))}{\dist_{\Var{I}}(x,x')},
\end{equation}
where $\dist_{\Var{I}}(\cdot,\cdot)$ denotes the distance given by the Riemannian metric on $\Var{I}$ and likewise for $\Var{O}$.
\changed{It follows immediately that the condition number yields an asymptotically sharp upper bound on the perturbation of the output of a computational problem when subjected to a perturbation of the input $x$. Indeed, we have
\[
 \dist_{\Var{O}}\left( f(x), f(x') \right) \le \kappa[f](x) \cdot \dist_\Var{I}( x, x' ) + o\left( \dist_\Var{I}( x, x' ) \right).
\]}%
The number $\kappa[f](x)$ is an \emph{absolute} condition number. On the other hand, \changed{if $\Var{I}$ is a submanifold of a vector space equipped with the norm $\|\cdot\|_{\Var{I}}$ and analogously for $\Var{O}$ and $\|\cdot\|_\Var{O}$, then} the \emph{relative} condition number is $\kappa_\mathrm{rel}[f](x):= \kappa[f](x) \,\tfrac{\Vert x \Vert_{\Var{I}}}{\Vert f(x)\Vert_{\Var{O}}}$. We focus on computing the absolute condition number, as the relative condition number can be inferred from the absolute one.

\subsection{Condition number of approximation problems}
\changed{Our starting point is Rice's classic definition \cref{computational_problem2}, which is often easy to compute for differentiable maps, rather than Demmel and Renegar's inverse distance to ill-posedness, which is generally more difficult to compute. Nevertheless, Rice's definition cannot be applied unreservedly to measure how sensitive a solution of \cref{eqn_PIP} is to perturbations of the ambient input $a\in\R^n$. At first glance, it seems that we could apply \cref{computational_problem2} to
\[
\phi : \R^n \to \Var{O},\; a \mapsto \argmin_{y\in\Var{O}} \Vert F(y) - a \Vert^2.
\]%
However, this is an ill-posed expression, as $\phi$ is not a map in general! Indeed, there can be several global minimizers for an ambient input $a\in\R^n$ and their number can depend on $a$. Even if we could restrict the domain so that $\phi$ becomes a map, the foregoing formulation has two additional limitations.
First, when \cref{eqn_PIP} is a nonconvex Riemannian optimization problem, it is unrealistic to expect that we can solve it globally. Instead we generally find local minimizers whose sensitivity to perturbations of the input is also of interest. For our theory to be practical, we need to consider the condition number of every local minimizer separately. Exceeding this goal, we will present the condition number of each critical point in this paper. 
Second,} requiring that the systems in \cref{eqn_PIP} can be modeled by a smooth map $F : \Var{O} \to \Var{I}$ excludes some interesting computational problems. For example, computing the eigenvalues of a symmetric matrix cannot be modeled as a map from outputs to inputs: there are infinitely many symmetric matrices with a prescribed set of eigenvalues. 

\changed{The solution to foregoing complications consists of allowing \textit{multivalued} or \textit{set-valued} maps $g : \Var{I} \rightrightarrows \Var{O}$ in which one input $x\in\Var{I}$ can produce a set of outputs $Y\subset\Var{O}$ \cite{AF2009,DR2009}. Such maps are defined implicitly by the graph 
\[
 G = \{ (x,y) \in \Var{I}\times\Var{O} \mid y \in g(x) \}.
\] 
In numerical analysis, Wilkinson \cite{Wilkinson1963} and Wo\'{z}niakowski \cite{Wozniakowski1976} realized that the condition number of \emph{single-valued localizations} or local maps of $G$ could still be studied using the standard techniques.
Blum, Cucker, Shub, and Smale \cite{BCSS} developed a geometric framework for studying the condition number of single-valued localizations when the graph $G$ is an algebraic subvariety. More recently, B\"urgisser and Cucker \cite{BC2013} considered condition numbers when the graph $G$ is a smooth manifold. Because these graphs are topologically severely restricted, the condition number at an input--output pair $(x,y)\in G$ can be computed efficiently using only linear algebra.}

\changed{The heart of this paper consists of characterizing the condition number of critical points of PIPs (and an implicit generalization thereof) by employing the geometric approach to conditioning from \cite{BCSS,BC2013}. Therefore, we summarize its main points next.}

\subsection{Implicit smooth computational problems}
Following \cite{BCSS,BC2013}, we model an inverse problem with input manifold $\Var I$ and output manifold $\Var O$ implicitly by a \emph{solution manifold}:
\[ %\label{computational_problem} 
\Var{S} \subset \Var{I}\times \Var{O}.
\]
Given an input $x\in \Var{I}$, the problem is finding a corresponding output $y \in\Var O$ such that $(x,y) \in\Var{S}$.
As $\Var{S}$ models a computational problem, we can consider its condition number.
Now it is defined at $(x,y)\in \Var{S}$ and not only at the input~$x$, because the sensitivity to perturbations in the output depends on \textit{which} output is considered.

\changed{The condition number of (the computational problem modeled by) $\Var{S}$ is derived} in \cite{BCSS,BC2013}. Let $\pi_{\Var{I}}:\Var{S} \to \Var{I}$ and $\pi_{\Var{O}}:\Var{S} \to \Var{O}$ be the projections on the input and output manifolds, respectively. \changed{Recall that \cite{BCSS,BC2013} assume} that
$\mathrm{dim} \,\Var{S} =\mathrm{dim}\, \Var{I}$,
because otherwise almost all inputs $x$ have infinitely many outputs ($\mathrm{dim}\, \Var{S} > \mathrm{dim} \,\Var{I}$) or no outputs ($\mathrm{dim}\, \Var{S} < \mathrm{dim} \,\Var{I}$). Furthermore, we will assume that~$\pi_{\Var{I}}$ is surjective, which means that every input~$x\in\Var{I}$ has at least one output.
When \changed{the derivative}\footnote{See \cref{sec_preliminaries} for a definition.} $\deriv{\pi_{\Var{I}}}{(x,y)}$ is invertible, the inverse function theorem \cite[Theorem 4.5]{Lee2013} implies that $\pi_{\Var{I}}$ is locally invertible. Hence, there is a local smooth \emph{solution map} $\pi_{\Var{O}} \circ \pi_{\Var{I}}^{-1}$, which locally around $(x,y)$ makes the computational problem explicit. This local map has condition number \changed{$\kappa[\pi_{\Var{O}} \circ \pi_{\Var{I}}^{-1}](x)$ from \cref{computational_problem2}. Moreover, because the map is differentiable, Rice gives an expression for \cref{computational_problem2} in terms of the spectral norm of the differential in \cite[Theorem 4]{Rice1966}}. \changed{In summary, \cite{BCSS,BC2013} conclude that the condition number of $\Var{S}$ at $(x,y)\in\Var{S}$ is
\begin{align}\nonumber
\kappa[\Var{S}](x,y)&=
\begin{cases} \kappa[\pi_{\Var{O}} \circ \pi_{\Var{I}}^{-1}](x)  & \text{ if $\deriv{\pi_{\Var{I}}}{(x,y)}$ is invertible},\\
\infty & \text{ otherwise}\end{cases}\\
\label{def_kappa}
&= \| (\deriv{\pi_\Var{O}}{(x,y)}) (\deriv{\pi_{\Var{I}}}{(x,y)})^{-1} \|_{\Var{I}\to\Var{O}},
\end{align}
where the latter \textit{spectral norm} is defined below in \cref{eqn_spectral_norm}.
Note that if $\Var{S}$ is the graph of a smooth map $f:\Var{I}\to \Var{O}$, then $\kappa[\Var{S}](x,f(x))=\kappa[f](x)$.}

\changed{\begin{remark}
 By definition, \cref{def_kappa} evaluates to $\infty$ if $\deriv{\pi_\Var{I}}{(x,y)}$ is not invertible.
\end{remark}}

The points in 
$\Var{W} := \{(x,y)\in \Var{S} \mid \kappa[\Var{S}](x,y) < \infty\}$ are called \emph{well-posed tuples}. Points in $\Var S\backslash \Var W$ are called \emph{ill-posed tuples}. An infinitesimal perturbation in the input $x$ of an ill-posed tuple $(x,y)$ can cause an arbitrary change in the output $y$. It is therefore natural to assume that $\Var W\neq \emptyset$, because otherwise the computational problem is always ill-posed. If there exists $(x,y)\in\Var W$, it follows from the inverse function theorem that $\Var{W}$ is an open submanifold of $\Var{S}$, so $\dim \Var S = \dim \Var W$. We also assume that $\Var{W}$ is dense. Otherwise, since $\Var{S}$ is Hausdorff, around any data point outside of $\overline{\Var W}$ there would be an \textit{open} neighborhood on which the computational problem is ill-posed; we believe these input-output pairs should be removed from the definition of the problem.

% \nick{Henceforth we consider systems that can be modeled implicitly via a solution manifold, generalizing the explicit map $F : \Var{O} \to \Var{I}$ in \cref{eqn_PIP}. In summary, we treat the following solution maps.}

In summary, we make the following assumption throughout this paper.

\begin{assumption}\label{ass_1}
The $m$-dimensional input manifold $\Var{I}$ is an embedded smooth submanifold of $\R^n$ equipped with the Euclidean inner product inherited from $\R^n$ as Riemannian metric. The output manifold~$\Var{O}$ is a smooth Riemannian manifold. The solution manifold $\Var{S} \subset \Var{I} \times \Var{O}$ is a smoothly embedded $m$-dimensional submanifold. The projection $\pi_{\Var{I}} : \Var{S} \to \Var{I}$ is surjective. The well-posed tuples $\Var{W}$ form an open dense embedded submanifold of~$\Var{S}$.
\end{assumption}

\changed{This assumption} is typically not too stringent. It is satisfied for example if $\Var{S}$ is the graph of a smooth map $f : \Var{I} \to \Var{O}$ or $F : \Var{O} \to \Var{I}$ (as for IPs) by \cite[Proposition 5.7]{Lee2013}. Another example is when $\Var{S}$ is the smooth locus of an algebraic subvariety of $\Var{I} \times \Var{O}$, as most problems in \cite{BC2013}.
Moreover, the condition of an input--output pair $(x,y)$ is a local property. Since every immersed submanifold is locally embedded \cite[Proposition 5.22]{Lee2013}, taking a restriction of the computational problem $\Var{S}$ to a neighborhood of $(x,y)$, in its topology as immersed submanifold, yields an embedded submanifold. The assumption and proposed theory always apply in \changed{this local sense}.

\subsection{Related work}
\changed{The condition number \cref{computational_problem2} was computed for several computational problems. Already in the 1960s, the condition numbers of solving linear systems, eigenproblems, linear least-squares problems, and polynomial rootfinding, among others were determined \cite{Wilkinson1963,Wilkinson1965,ANLA,matrix_computations,BC2013}. An astute observation by Demmel \cite{Demmel1987,Demmel1987a} in 1987 paved the way for the introduction of condition numbers in optimization. He realized that many condition numbers from linear algebra can be interpreted as an inverse distance to the nearest ill-posed problem. This characterization generalizes to contexts where an explicit map as required in \cref{computational_problem2} is not available.  Renegar \cite{Renegar1995a,Renegar1995b} exploited this observation to determine the condition number of linear programming instances and he connected it to the complexity of interior-point methods for solving them. This ignited further research into the sensitivity of optimization problems, leading to (Renegar's) condition numbers for linear programming \cite{Renegar1995a,Renegar1995b,CC2001}, conic optimization \cite{EF2002,CCP2008,PR2020,BF2009,Vera2014}, and convex optimization \cite{CC2009,AB2012,FO2005} among others. In many cases, the computational complexity of methods for solving these problems can be linked to the condition number; see \cite{Vera1996,CP2001,FV1999,Renegar1995a,Renegar1995b,EF2000}. In our earlier work \cite{BV2017b}, we connected the complexity of Riemannian Gauss--Newton methods for solving \cref{eqn_PIP} to the condition number \textit{of the exact inverse problem} $F(y)=x$. We anticipate that the analysis from \cite{BV2017b} could be refined to reveal the dependence on the condition number of the optimization problem \cref{eqn_PIP}.} 

% the Lipschitz constant\footnote{This is also known as the Hoffman constant in the case of polyhedral convex maps \cite[Lemma 3C.4]{DR2009}} \cite[Chapter 3C]{DR2009}, outer Lipschitz constant \cite[Chapter 3D]{DR2009}, and
\changed{Beside Renegar's condition number, other measures of sensitivity are studied in optimization. The sensitivity of a set-valued map $g : \Var{I} \rightrightarrows \Var{O}$ at $(x,y)$ with $y\in g(x)$ is often measured by the \textit{Lipschitz modulus} \cite[Chapter 3E]{DR2009}. This paper only studies the sensitivity of a single-valued localization $\hat{g}$ of the set-valued map $g$, in which case the Lipschitz modulus $K$ of $g$ at $(x,y)$ is a Lipschitz constant of $\hat{g}$ \cite[Proposition 3E.2]{DR2009}. Consequently, Rice's condition number $\kappa[\hat{g}](x)$, which we characterize in this paper for Riemannian approximation problems, is always upper bounded by $K$; see, e.g., \cite[Proposition 6.3.10]{HJE2017}.}

\subsection{Main contributions}
We characterize the condition number of computing global minimizers, local minimizers, and critical points of the following implicit generalization of \cref{eqn_PIP}:
\begin{align}\tag{IPIP}\label{eqn_IPIP}
 \pi_{\Var{O}} \circ \argmin_{(x,y)\in\Var{S}} \frac{1}{2} \| a - \pi_{\Var{I}}(x,y) \|^2,
\end{align}
where $a \in \R^n$ is a given ambient input and the norm is the Euclidean norm. We call this implicit formulation of \cref{eqn_PIP} a Riemannian implicit PIP.

\changed{Our main contributions are \cref{main2,main11,main3}. The condition number of critical points of  the Riemannian approximation problem $\min_{x\in\Var{O}} \| a - x \|^2$, is treated in \cref{main2}. 
\Cref{main11} contains the condition number of \cref{eqn_IPIP} at global minimizers when $a$ is sufficiently close to the input manifold. Finally, the most general case is \cref{main3}: it gives the condition number of \cref{eqn_IPIP} at all critical points including local and global minimizers}.

Aforementioned theorems will show that the way in which the input manifold $\Var{I}$ \textit{curves} in its ambient space $\R^n$, as measured by classic differential invariants, enters the picture.
To the best of our knowledge, the role of curvature in the theory of condition of smooth computational problems involving approximation was not previously identified. We see this is as one of the main contributions of this article.

\subsection{Approximation versus idealized problems}
The results of this paper raise a pertinent practical question: If one has a computational problem modeled by a solution manifold $\Var{S}$ that is solved via an optimization formulation as in \cref{eqn_IPIP}, which condition number is appropriate? \changed{The one of the idealized problem $\Var{S}$, where the inputs are constrained to $\Var{I}$? Or the more complicated condition number of the optimization problem, where we allow ambient inputs from $\Var{I}$'s ambient space $\R^n$}?
We argue that the choice is only an illusion.

The first approach should be taken if the problem is \textit{defined} implicitly by a solution manifold $\Var S\subset \Var I\times \Var O$ and the ambient inputs are perturbations $a\in\Var \R^n$ of inputs $x\in\Var{I}$. The idealized framework will still apply because of \cref{main1}.
Consider for example the $n$-camera triangulation problem, discussed in \cref{sec_triangulation}, where the problem $\Var S_\mathrm{MV}$ is defined by computing the inverse of the camera projection process.
If the measurements $a$ were taken with a perfect pinhole camera for which the theoretical model is exact, then even when $a\not \in\Var{I}_\mathrm{MV}$ (e.g., errors caused by pixelization), as long as it is a sufficiently small perturbation of some $x\in\Var I_\mathrm{MV}$, the first approach is suited.

On the other hand, if it is only \textit{postulated} that the computational problem can be well approximated by the \textit{idealized problem} $\Var{S}$, then the newly proposed framework for Riemannian approximation problems is the more appropriate choice. In this case, the ``perturbation'' to the input is a modeling error.
An example of this is applying $n$-camera triangulation to images obtained by cameras with lenses, i.e., most real-world cameras.
In this case, the computational problem is not $\Var{S}_\mathrm{MV}$, but rather a more complicated problem that takes into account lens distortions. When using $\Var{S}_\mathrm{MV}$ as a proxy, as is usual \cite{FL2001, HZ2003, Maybank1993}, the computational problem becomes an optimization problem. Curvature should be taken into account then, as developed here.

% \subsection{Other notions of sensitivity}
% \todo{Metric regularity of set-valued maps ... yeah ridiculous to consider metric regularity from our point of view. We are interested in the sensitivity of the localizations. This Pompeiu--Haussdorf distance used in metric regularity yields a Lipschitz constant for the worst thing that could happen, jumping from one solution to a completely different one. This is not the setting that we consider.}
% 
% \todo{Mention main theorem about the localization. This implies that local minimizer with finite condition number will not jump, so that metric regularity is typically not the quantity that best explains how minimizers move under perturbation of the data.}

\subsection{Outline}
Before developing the main theory, we briefly recall the necessary concepts from differential geometry \changed{and fix the notation in the next section. A key role will be played by the Riemannian Hessian of the distance from an ambient input in $\R^n$ to the input manifold $\Var{I}$. Its relation to classic differential-geometric objects such as the Weingarten map is stated in \cref{sec_weingarten}. The condition number of critical points of \cref{eqn_IPIP} is first characterized in \cref{sec_results_2} for the special case where the solution manifold represents the identity map. The general case is studied in \Cref{sec_results_1,sec_results_3}; the former deriving the condition number of global minimizers, while the latter treats critical points.}
The expressions of the condition numbers in \cref{sec_results_1,sec_results_2,sec_results_3} are obtained from implicit formulations of the Riemannian IPIPs. 
For clarity, in the case of local minimizers, we also express them as condition numbers of global minimizers of localized Riemannian optimization problems in \cref{sec_optimization}. 
\changed{\Cref{sec_computational} explains a few standard techniques for computing the condition number in practice along with their computational complexity. Finally,} an application of our theory to the triangulation problem in multiview geometry is presented in \cref{sec_triangulation}.

\section{Preliminaries}  \label{sec_preliminaries}

We briefly recall the main concepts from Riemannian geometry that we will use. The notation introduced here will be used throughout the paper.
Proofs of the fundamental results appearing below can be found in \cite{Lee1997,Lee2013,ONeill1983,ONeill2001,riemannian_geometry,Petersen}.

\subsection{Manifolds} \label{sec_sub_manifold}
By \textit{smooth $m$-dimensional manifold} $\Var{M}$ we mean a $C^\infty$ topological manifold that is second-countable, Hausdorff, and locally Euclidean of dimension~$m$. The tangent space $\Tang{p}{\Var{M}}$ is the $m$-dimensional linear subspace of differential operators at $p$. A differential operator at $p\in\Var{M}$ is a linear map $v_p$ from the vector space $C^\infty(\Var{M})$ over $\R$ of smooth functions $f : \Var{M} \to \R$ to $\R$ that satisfies the product rule:
\(
 v_p (f \cdot g) = (v_p f) \cdot g(p) + f(p) \cdot (v_p g)
\)
for all $f, g \in C^\infty(\Var{M})$. If $\Var{M} \subset \R^n$ is \textit{embedded}, the tangent space can be identified with the linear span of all tangent vectors $\frac{\mathrm{d}}{\mathrm{d}t}|_{t=0} \gamma(t)$ where $\gamma \subset \Var{M}$ is a smooth curve passing through $p$ at $0$. 
% The identification of $\frac{\mathrm{d}}{\mathrm{d}t}|_{t=0} \gamma(t)$
% as a differential operator is $v_1 \partial_1 + \cdots + v_n \partial_n$, where $\frac{\mathrm{d}}{\mathrm{d}t}|_{t=0} \gamma(t) \gamma(t)=(v_1,\ldots,v_n)$ and where $\partial_i x_j = \delta_{i,j}$, the Kronecker delta, and $x_j$ are coordinates on $\R^n$.

A differentiable map $F : \Var{M} \to \Var{N}$ between manifolds $\Var{M}$ and $\Var{N}$ induces a linear map between tangent spaces $\deriv{F}{p}:\Tang{p}{\Var{M}}\to \Tang{F(p)}{\Var{M}}$, called the \emph{derivative} of $F$ at $p$. If $v_p\in \Tang{p}{\Var{M}}$, then $w_{F(p)} := (\deriv{F}{p})(v_p)$
is the differential operator $w_{F(p)}(f) := v_p(f\circ F)$ for all $f\in C^{\infty}(\Var N)$.

The identity map $\mathbf{1}_\Var{M} : \Var{M} \to \Var{M}$ will be denoted by $\mathbf{1}$, the space on which it acts being clear from the context.

\subsection{Tubular neighborhoods}\label{sec_sub_bundles}
Let $\Var{M}^m \subset \R^n$ be an embedded submanifold. The disjoint union of tangent spaces to $\Var{M}^m$ is the \textit{tangent bundle} of $\Var{M}$:
\[
 \mathrm{T}\Var{M} := \coprod_{p\in\Var{M}} \Tang{p}{\Var{M}} = \{ (p, v) \mid p \in \Var{M} \text{ and } v \in \Tang{p}{\Var{M}} \}.
\]
The tangent bundle is a smooth manifold of dimension $2m$. The \textit{normal bundle} is constructed similarly. The \textit{normal space} of $\Var{M}$ in $\R^n$ at $p$ is the orthogonal complement in $\R^n$ of the tangent space $\Tang{p}{\Var{M}}$, namely $\mathrm{N}_p \Var{M} := ( \Tang{p}{\R^n} )^\perp$. The normal bundle is a smooth embedded submanifold of $\mathbb{R}^n \times \mathbb{R}^n$ of dimension $n$. %; see \cite[Theorem 6.23]{Lee2013}.
Formally, we write
\[
 \mathrm{N}\Var{M} := \coprod_{p \in \Var{M}} \mathrm{N}_{p} \Var{M} = \{ (p, \eta) \mid p \in \Var{M} \text{ and } \eta \in \mathrm{N}_{p} \Var{M} \}.
\]
Let $\delta : \Var{M} \to \R_+$ be a positive, continuous function. Consider the following open neighborhood of the normal bundle:
\begin{equation}\label{V_delta}
 \Var{V}_\delta = \{ (p, \eta_p) \in \mathrm{N}\Var{M} \mid \|\eta_p\| < \delta(p) \},
\end{equation}
where the norm is the one induced from the Riemannian metric; in our setting it is the standard norm on $\R^n$. The map $E : \mathrm{N}\Var{M} \to \R^n, (p,\eta) \mapsto p + \eta$ is smooth and there exists a $\delta$ such that the restriction $E|_{\Var{V}_\delta}$ becomes a diffeomorphism onto its image \cite[Theorem 6.24]{Lee2013}. Consequently, $\Var{T} = E(\Var{V}_\delta)$ is an $n$-dimensional, open, smooth, embedded submanifold of $\R^n$ that forms a neighborhood of $\Var{M}$. The submanifold $\Var{T}$ is called a \textit{tubular neighborhood} of $\Var{M}$ \changed{and $\delta$ is called its \textit{height} (function)}.

\subsection{Vector fields}\label{sec_sub_connection}
A \textit{smooth vector field} is a smooth map $X : \Var{M} \to \mathrm{T}\Var{M}$ from the manifold to its tangent bundle taking a point $p$ to a tangent vector $v_p$. By the notation $X|_p$ we mean $X(p)$. A smooth vector field $X$ on a properly embedded submanifold $\Var{M} \subset \Var{N}$ can be extended to a smooth vector field $\widehat{X}$ on $\Var{N}$ such that it agrees on $\Var{M}$: $\widehat{X}|_\Var{M} = X$. 
% A vector field $X$ can be multiplied by a function $f$ as follows $(fX)|_p := f(p) X|_p$.

A \textit{smooth frame} of $\Var{M}^m$ is a tuple of $m$ smooth vector fields $(X_1, \ldots, X_m)$ that are linearly independent: $(X_1|_p, \ldots, X_m|_p)$ is linearly independent for all $p \in \Var{M}$. If these tangent vectors are orthonormal for all $p$, then the frame is called orthonormal.

Let $X$ be a smooth vector field on $\Var{M}$. An \textit{integral curve} of $X$ is a smooth curve $\gamma\subset\Var{M}$ such that $\deriv{\gamma}{t} = X|_{\gamma(t)}$ for all $t \in (-1,1)$. For every $p \in \Var{M}$ the vector field $X$ generates an integral curve $\gamma$ with starting point $p = \gamma(0)$. There is always a smooth curve \textit{realizing} a tangent vector $v_p\in \Tang{p}{\Var{M}}$, i.e., a curve $\gamma$ with $p = \gamma(0)$ and $v = \deriv{\gamma}{0}$.

\subsection{Riemannian manifolds}

A \textit{Riemannian manifold} is a smooth manifold $\Var{M}$ equip\-ped with a \textit{Riemannian metric} $g$: a positive definite symmetric bilinear form $g_p:\Tang{p}{\Var{M}}\times \Tang{p}{\Var{M}} \to \R$ that varies smoothly with $p$. The metric $g_p$ induces a norm on $\Tang{p}{\Var{M}}$ denoted $\Vert t_p \Vert_{\Var M} := \sqrt{g_p(t_p, t_p)}.$
The only Riemannian metric we will explicitly use in computations is the standard Riemannian metric of $\R^n$, i.e., the Euclidean inner product $g_p(x,y) = \langle x, y\rangle = x^T y$.

If $\Var{M}$ and $\Var N$ are Riemannian manifolds with induced norms $\Vert \cdot\Vert_{\Var M}$ and $\Vert \cdot \Vert_{\Var N}$, and if $F :\Var{M}\to \Var{N}$ is a differentiable map, then the \textit{spectral norm} of $\deriv{F}{p}$ is
\begin{align}\label{eqn_spectral_norm}
\| \deriv{F}{p} \|_{\Var{M}\to\Var{N}} := \sup_{t_p \in \Tang{p}{\Var{M}}\setminus \{0\}}\frac{\Vert (\deriv{F}{p})(t_p) \Vert_{\Var N}}{\Vert  t_p\Vert_{\Var M}}.
\end{align}
If the manifolds are clear from the context we can drop the subscript and write $\| \deriv{F}{p} \|$.

\subsection{Riemannian gradient and Hessian}

\changed{Let $\Var{M}$ be a Riemannian manifold with metric $g$, and let $f : \Var{M} \to \R$ be a smooth function. The \textit{Riemannian gradient} of $f$ at $p\in\Var{M}$ is the tangent vector $\nabla f|_p \in \Tang{p}{\Var{M}}$ that is dual to the derivative $\deriv{f}{p} : \Tang{p}{\Var{M}} \to \R$ under $g$, i.e., $g_p(\nabla f|_p, t) = (\deriv{f}{p})(t)$ for all $t\in\Tang{p}{\Var{M}}$. It is known that $\nabla f$ is a vector field, hence explaining the notation $\nabla f|_p$ for its value at $p\in\Var{M}$.}

\changed{When $\Var{M}\subset\R^n$ is a Riemannian embedded submanifold with the inherited Euclidean metric, the \textit{Riemannian Hessian} \cite[Definition 5.5.1]{AMS2008} can be defined without explicit reference to connections, as in \cite[Chapter 5]{Boumal2020}:
\[
 \mathrm{Hess}_p (f) : \Tang{p}{\Var{M}} \to \Tang{p}{\Var{M}},\; \eta_p \mapsto \mathrm{P}_{\Tang{p}{\Var{M}}} \left( (\deriv{\nabla f}{p})(\eta_p) \right),
\]
where $\nabla f$ now has to be viewed as a vector field on $\R^n$ and $\mathrm{P}_{\Tang{p}{\Var{M}}}$ projects orthogonally onto the tangent space $\Tang{p}{\Var{M}}$. In other words, the Hessian takes $\eta_p$ to $\mathrm{P}_{\Tang{p}{\Var{M}}} \left( \frac{\mathrm{d}}{\mathrm{dt}}|_{t=0} {\nabla f}|_{\gamma(t)} \right)$
where $\gamma(t) \subset \Var{M}$ is a curve realizing the tangent vector $\eta_p$.}

\changed{The Hessian $\mathrm{Hess}_p(f)$ is a symmetric bilinear form on the tangent space $\Tang{p}{\Var{M}}$ whose (real) eigenvalues contain essential information about the objective function $f$ at $p$. If the Hessian is positive definite, negative definite, or indefinite at a \textit{critical point} $p$ of $f$ (i.e., $\nabla f|_p = 0$ or, equivalently, $\deriv{f}{p}=0$), then $p$ is respectively a local minimum, local maximum, or saddle point of $f$.}

\section{Measuring curvature with the Weingarten map} \label{sec_weingarten}

\changed{A central role in this paper is played by the Riemannian Hessian of the 
distance function
\begin{align}\label{eqn_distance_function}
d : \Var{I} \to \R,\; x \mapsto \frac{1}{2}\|x - a\|^2,
\end{align}
which measures the Euclidean distance between a fixed point $a \in \R^n$ and points on the embedded submanifold $\Var{I} \subset \R^n$. It is the inverse of this Riemannian Hessian matrix that will appear as an additional factor to be multiplied with $\deriv{\pi_\Var{O}}{(x,y)} (\deriv{\pi_\Var{I}}{(x,y)})^{-1}$ from \cref{def_kappa} in the condition numbers of the problems studied in \cref{sec_results_1,sec_results_2,sec_results_3}.

The Riemannian Hessian of $d$ contains classic differential-geometric information about the way $\Var{I}$ curves inside of $\R^n$ \cite{Petersen}. Indeed, from \cite[Equation (10)]{AMT2013} we conclude that it can be expressed at $x\in\Var{I}$ as
\[
\mathrm{Hess}_x (d) = \mathbf{1}_{\Tang{x}{\Var{I}}} + \SFF_x\left( \mathrm{P}_{\NSp{x}{\Var{M}}} (x-a) \right),
\]
where $\mathrm{P}_{\NSp{x}{\Var{M}}}$ is the projection onto the normal space, and
$\SFF_x : \NSp{x}{\Var{M}} \to S^2( \Tang{x}{\Var{M}} )$ is a classical differential invariant called the \textit{second fundamental form} \cite{Lee1997,ONeill1983,ONeill2001,riemannian_geometry,Petersen}, which takes a normal vector in $\NSp{x}{\Var{M}}$ and sends it to a symmetric bilinear form on the tangent space $\Tang{x}{\Var{M}}$. If we let $\eta_x := a - x$, then at a critical point of $d$ we have $\eta_x \in \NSp{x}{\Var{I}}$, so that the Hessian of interest satisfies
\begin{equation}\tag{H}\label{eqn_Hess_distance}
 H_{\eta} := \mathrm{Hess}_x(d) = \mathbf{1} - S_{\eta},\quad\text{where } S_\eta := \SFF_x(\eta_x).
\end{equation}
Note that we dropped the base point $x\in\Var{I}$ from the notation in $H_\eta$ and $S_\eta$. In the above, $S_\eta$ is interpreted as a symmetric endomorphism on $\Tang{x}{\Var{I}}$ and it is classically called the \textit{shape operator} or \textit{Weingarten map} \cite{Lee1997,ONeill1983,ONeill2001,riemannian_geometry,Petersen}.}

\changed{The Weingarten map measures how much the embedding into the ambient space $\R^n$ bends the manifold $\Var{I}^m$ in the following concrete way \cite[Chapter~8]{Lee1997}.}
For $\eta\neq 0$, let $w:=\tfrac{\eta}{\Vert \eta\Vert}$ and let $c_1,\ldots,c_m$ be the eigenvalues of~$S_w = \tfrac{1}{\Vert \eta\Vert} S_\eta$. As $S_w$ is self-adjoint, all $c_i$'s are all real. The eigenvalues of~$S_\eta$ are then $c_1\Vert \eta\Vert ,\ldots,c_m\Vert \eta\Vert$. The $c_i$'s are called the \emph{principal curvatures} of $\Var{I}$ at $x$ in direction of~$\eta$. They have the following classic geometric meaning \cite[Chapter 1]{SpivakVol2}: If $u_i\in \mathrm{T}_x \Var{I}$ is a unit length eigenvector of the eigenvalue~$c_i$,
then locally around $x$ and in the direction of $u_i$ the manifold $\Var{I}$ contains an infinitesimal arc of a circle with center $x+c_i^{-1}w =x+(c_i\Vert \eta\Vert)^{-1}\eta$ through $x$. This circle is called an \emph{osculating circle}. The radii $r_i := \vert c_i\vert^{-1}$ of those circles are called the \emph{critical radii} of~$\Var{I}$ at $x$ in direction $\eta$. See also \cref{fig3}.

\changed{For several submanifolds $\Var{I}\subset\R^n$ relevant in optimization applications, the Weingarten map $S_\eta$ was computed. Absil, Mahoney, and Trumpf \cite{AMT2013} provide expressions for the sphere, Stiefel manifold, orthogonal group, Grassmann manifold, and low-rank matrix manifold. Feppon and Lermusiaux \cite{FL2019} further extended these results by computing the Weingarten maps and explicit expressions for the principal curvatures of the isospectral manifold and ``bi--Grassmann'' manifold, among others.
Heidel and Schultz \cite{HS2018} computed the Weingarten map of the manifold of higher-order tensors of fixed multilinear rank \cite{HRS2012}.}

\section{The condition number of critical point problems} \label{sec_results_2}
% In applications involving the problem \cref{eqn_IPIP}, the projection is often approximated via some optimization procedure applied to a distance function, as in the example in \cref{sec_triangulation}. For nonconvex manifolds $\Var{I}$, optimization methods usually only guarantee that the first-order optimality conditions are satisfied \cite{AMS2008}. This made us wonder about the condition number of \cref{eqn_ap} where we replace the projection $\mathrm{P}_\Var{I}$ by a map that produces one of the points satisfying these first-order optimality conditions, further generalizing the \cref{eqn_ap} to a computational problem that we call the \emph{generalized critical point problem} (GCPP). This most general formulation is discussed in the next section. Here, we study a special, easier case.

\changed{Prior to treating \cref{eqn_IPIP} in general, we focus on an important and simpler special case.}
This case arises when the solution manifold $\Var{S} \subset \Var{I} \times \Var{O}$ is the graph of the identity map $\mathbf{1} : \Var{I} \to \Var{I}$, so $\Var{O} =\Var{I}$. Given an ambient input $a \in \R^n$, the problem \cref{eqn_IPIP} now comprises finding an output $x \in \Var{I}$ so that \changed{the distance function $d$ from \cref{eqn_distance_function}} is minimized. This is can be formulated as the standard Riemannian least-squares problem 
\begin{align}\label{eqn_basic_riemannian}
 \min_{x \in \Var{I}} \frac{1}{2} \| x - a \|^2.
\end{align}
Computing \textit{global} minimizers of this optimization problem is usually too ambitious a task when $\Var{I}$ is a complicated nonlinear manifold. Most Riemannian optimization methods \changed{will converge to local minimizers by seeking} to satisfy the first-order necessary optimality conditions of \cref{eqn_basic_riemannian}, namely $a-x\in\mathrm{N}_x \Var{I}$. We say that such methods attempt to solve the \textit{critical point problem} (CPP) whose implicit formulation is
\begin{align} \label{eqn_cpp} \tag{CPP}
\Var{S}_\mathrm{CPP} := \left\{ (a,x) \in \mathbb{R}^n \times \Var{I}  \mid  a - x \in \mathrm{N}_{x} \Var{I}\right\}.
\end{align}

We call $\Var{S}_\mathrm{CPP}$ the \textit{critical point locus}. \changed{It is an exercise to show that $\Var{S}_\mathrm{CPP}$ is linearly diffeomorphic to the normal bundle $\mathrm{N}\Var{I}$ so that we have the following result.}

\begin{lemma} \label{C_is_a_manifold}
$\Var{S}_\mathrm{CPP}$ is an embedded submanifold of $\R^n\times \Var{I}$ of dimension $n$.
\end{lemma}

Consequently, the CPP falls into the realm of the theory of condition from~\cref{def_kappa}. We denote the coordinate projections of $\Var{S}_\mathrm{CPP}$ by $\Pi_{\R^n}: \Var{S}_\mathrm{CPP} \to \mathbb{R}^n$ and
 $\Pi_\Var{I}: \Var{S}_\mathrm{CPP} \to \Var{I}$ for distinguishing them from the projections associated to $\Var{S}$. The corresponding submanifold of well-posed tuples is
\[
 \Var{W}_\mathrm{CPP} := \{ (a,x) \in \Var{S}_\mathrm{CPP}  \mid \deriv{\Pi_{\R^n}}{(a,x)} \text{ is invertible} \}.
\]
The fact that it is a manifold is the following result, proved in \cref{sec_proofs_thm3}.
\begin{lemma}\label{lem_WCPP_is_manifold}
 $\Var{W}_\mathrm{CPP}$ is an open dense embedded submanifold of $\Var{S}_\mathrm{CPP}$.
\end{lemma}

\changed{We show in \cref{sec_proofs_thm3} that the condition number of this problem can be characterized succinctly as follows.\footnote{\changed{We take the convention that whenever an expression $\|\cdot\|_{\Var{I}\to\Var{O}}$ is ill-posed it evaluates to $\infty$.}}}

\begin{theorem}[Critical points]\label{main2}
Let $(a,x)\in \Var{S}_\mathrm{CPP}$ and $\eta:=a-x$. Then,
\[
\kappa[\Var{S}_{\mathrm{CPP}}](a,x) = \Vert H_{\eta}^{-1}  \Vert_{\Var{I}\to\Var{I}},
\]
where $H_\eta$ is given by \cref{eqn_Hess_distance}. Moreover, if $(a,x) \in \Var{S}_\mathrm{CPP}\setminus\Var{W}_\mathrm{CPP}$, then $H_{\eta}$ is not invertible.
\end{theorem}

\begin{corollary}
$\kappa[\Var{S}_\mathrm{CPP}] : \Var{S}_\mathrm{CPP} \to \R \cup \{+\infty\}$ is continuous \changed{and $\kappa[\Var{S}_\mathrm{CPP}](x,y)$ is finite if and only if $(x,y)\in\Var{W}_\mathrm{CPP}$}.
\end{corollary}

Curvature has entered the picture in \cref{main2} in the form of \changed{the Riemannian Hessian $H_\eta$ from \cref{eqn_Hess_distance}.}
It is an \textit{extrinsic} property of $\Var{I}$, however, so that the condition number depends on the Riemannian embedding of $\Var{I}$ into $\R^n$. Intuitively this is coherent, as different embeddings give rise to different computational problems each with their own sensitivity to perturbations in the ambient space.

Let us investigate the meaning of the foregoing theorem a little further. \changed{Recall, from \cref{sec_weingarten}, the definition of the principal curvatures $c_i$ as the eigenvalues of the Weingarten map $S_{w}$ where $w = \frac{\eta}{\|\eta\|}$ and $\eta=a-x$. Then,} \cref{main2} implies that
\begin{equation}\label{formula_for_kappa_CPP}
\kappa[\Var{S}_{\mathrm{CPP}}](a,x) = \max_{1\leq i\leq m} \frac{1}{\big\vert 1-c_i\Vert \eta\Vert\big\vert},
\end{equation}
\changed{so its condition is completely determined by the amount by which $\Var{I}$ curves inside of $\R^n$ and the distance of the ambient input $a$ to the critical point $x\in\Var{I}$.}
\changed{We see that} $\kappa[\Var{S}_{\mathrm{CPP}}](a,x) = \infty$ if and only if $\Var{I}$ contains an infinitesimal arc of the circle with center $x+\eta = a$ \changed{(so that there is an $i$ with $c_i \Vert \eta\Vert = 1$). Consequently,} if $\|\eta\|$ lies close to any of the critical radii $r_i = |c_i|^{-1}$ then the CPP is ill-conditioned at $(a,x)$.

The formula from \cref{formula_for_kappa_CPP} makes it apparent that for some inputs we may have that $\kappa[\Var{S}_{\mathrm{CPP}}](a,x)<1$. In other words, for some ambient inputs $a\in \mathbb{R}^n$ the infinitesimal error $\Vert \Delta a\Vert$ can shrink! \Cref{fig3} gives a geometric explanation of when the error is amplified (panel (a)) and when it is shrunk (panels (b) and (c)).

\begin{figure}
  \begin{minipage}{0.32\textwidth}
    \begin{tikzpicture}[scale = 1.54]
      \draw[black]   plot[smooth,domain=-1.2:1.2] (\x, {(\x)*(\x)});
      \draw[black]  (0,1/2) circle[radius=1/2];

      \draw[black]  (0,0) circle[radius=1.25pt] node[below, xshift=5pt] {$x$};
      \fill[white]  (0,0) circle[radius=1.24pt];
      \draw[->] (0,0.3)  -- (0,0.1);

      \fill[gray]  (0,1/2) circle[radius=1.25pt];
      \fill[black]  (0,0.4) circle[radius=1.25pt] node[right] {$a$};

      \draw[->] (-0.1,0.4)  -- (-0.19,0.4)  node[left, fill=white] {$a+\Delta a$};
      \draw[->] (-0.2,0.35)  -- (-0.3,0.1125);
      \draw[->] (-0.1,0)  -- (-0.35,0) node[left] {$x+\Delta x$};
      \draw (0,-1/2) node {$\phantom{\Delta a}$};
 \end{tikzpicture}
 \centering (a)
 \end{minipage}
 \begin{minipage}{0.32\textwidth}
     \begin{tikzpicture}[scale = 1.54]
        \draw[black]   plot[smooth,domain=-3:-0.6] (\x, {(\x+1.8)*(\x+1.8)});
        \draw[black]  (-1.8,1/2) circle[radius=1/2];
        
        \fill[gray]  (-1.8,1/2) circle[radius=1.25pt];

        \draw[black]  (-1.8,0) circle[radius=1.25pt] node[below, xshift=7pt] {$x$};
        \fill[white]  (-1.8,0) circle[radius=1.24pt];
        \draw[->] (-1.8,-0.4)  -- (-1.8,-0.1);
        \fill[black]  (-1.8,-1/2) circle[radius=1.25pt] node[right] {$a$};

        \draw[->] (-1.9,-1/2)  -- (-2.3,-1/2)  node[left] {$a+\Delta a$};
        \draw[->] (-2.3,-0.4)  -- (-2.15,-0.1);
        \draw[->] (-1.9,0)  -- (-2.1,0) node[left] {$x+\Delta x$};
\end{tikzpicture}
\centering (b)
 \end{minipage}
 \begin{minipage}{0.32\textwidth}
\begin{tikzpicture}[scale = 1.54]
        \draw[black]   plot[smooth,domain=0.6:3] (\x, {(\x-1.8)*(\x-1.8)});
        \draw[black]  (1.8,1/2) circle[radius=1/2];
        \fill[gray]  (1.8,1/2) circle[radius=1.25pt];

        \draw[black]  (1.8,0) circle[radius=1.25pt] node[below, xshift=-5pt] {$x$};
        \fill[white]  (1.8,0) circle[radius=1.24pt];
        \draw[->] (1.8,1.1)  -- (1.8,0.1);
        \fill[black]  (1.8,5/4) circle[radius=1.25pt] node[right] {$a$};

        \draw[->] (1.7,5/4)  -- (1.4,5/4)  node[left, fill=white] {$a+\Delta a$};
        \draw[->] (1.4,1.1)  -- (2.1,0.2);
        \draw[->] (1.9,0)  -- (2.1,0) node[right] {$x+\Delta x$};
        \draw (0.6,-1/2) node {$\phantom{\Delta a}$};
 \end{tikzpicture}
 \centering (c)
 \end{minipage}
\caption{\label{fig3} The graphs shows ambient inputs $a$ for the CPP of the parabola. The osculating circle with critical radius lies above the parabola, and its center (the gray point) is a focal point of the parabola, i.e., it is a point in $\Sigma_\mathrm{CPP} = \Var{S}_\mathrm{CPP} \setminus \Var{W}_\mathrm{CPP}$. 
  \changed{In graph (a), the ambient input} $a$ also lies above the parabola, so $\eta=a-x$ points towards the focal point. This means that the principal curvature at $x$ in direction of $\eta$ is $c_1>0$. Moreover, $c_1\Vert \eta\Vert <1$, and thus, by \cref{formula_for_kappa_CPP}, we have $\kappa[\Var{S}_{\mathrm{CPP}}](a,x)>1$. In other words, the curvature of the parabola amplifies the perturbation $\Vert \Delta a\Vert$. \changed{In graph (b), the ambient input} $a$ lies below the parabola, so $\eta=a-x$ points away from the focal point. This means that the principal curvature at $x$ in direction of $\eta$ is negative, so that by \cref{formula_for_kappa_CPP} we have $\kappa[\Var{S}_{\mathrm{CPP}}](a,x)<1$. The curvature of the parabola shrinks the perturbation $\Vert \Delta a\Vert$. \changed{In graph (c), the ambient input} $a$ lies above the parabola, so that the corresponding critical curvature is positive. However, now $c_1\Vert \eta\Vert >1$, and so by \cref{formula_for_kappa_CPP} we have $\kappa[\Var{S}_{\mathrm{CPP}}](a,x)<1$. Again, the curvature of the parabola shrinks the perturbation $\Vert \Delta a\Vert$.} 
\end{figure}
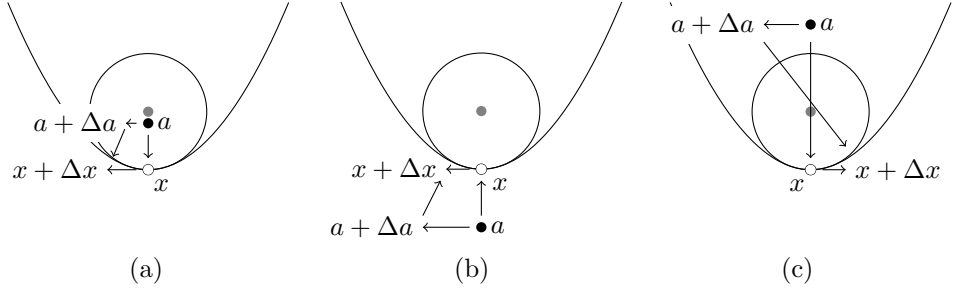

We also have an interpretation of $\kappa[\Var{S}_{\mathrm{CPP}}](a,x)$ as a normalized inverse distance to ill-posedness: for $(a,x)\in \Var{S}_\mathrm{CPP}$, let $\eta=a-x$ be the normal vector that connects~$x$ to $a$.
The locus of ill-posed inputs for the CPP is $\Sigma_\mathrm{CPP} = \Var{S}_\mathrm{CPP}\setminus\Var{W}_\mathrm{CPP}$. The projection $\Pi_{\R^n}(\Sigma_\mathrm{CPP})$ is well studied in the literature. Thom \cite{Thom1962} calls it the \emph{target envelope} of $\Var{S}_\mathrm{CPP}$, while Porteous \cite{PORTEOUS71} calls it the \emph{focal set} of $\Var{I}$, and for curves in the plane it is called the \emph{evolute}; see Chapter 10 of \cite{ONeill1983}. 
% \footnote{Actually, Thom calls it the target envelope of the normal bundle of $\Var{I}$. However, the normal bundle of $\Var{I}$ is diffeomorphic to $\Var{S}_\mathrm{CPP}$.}
The points in the intersection of $\{x\} \times (x+\R \eta)$ with $\Sigma_\mathrm{CPP}$ are given by the multiples of the normal vector $\eta$ whose length equals one of the critical radii $r_1,\ldots,r_m$; this is also visible in \cref{fig_pert_normal_direction}. Comparing with~\cref{formula_for_kappa_CPP} we find
\begin{equation*}
  \frac{1}{\kappa[\Var{S}_{\mathrm{CPP}}](a,x)} =  \min_{\eta_0\in \R \eta: \, (x,a+\eta_0) \in \Sigma_\mathrm{CPP}} \frac{\|\eta_0\|}{\vert \,\Vert \eta_0\Vert - \Vert \eta\Vert \, \vert},
\text{ where } \eta = a-x.\end{equation*}
This is an interpretation of Rice's condition number in the spirit of Demmel \cite{Demmel1987,Demmel1987a} \changed{and Renegar \cite{Renegar1995a,Renegar1995b}.}

\begin{remark}
If $\Var{I}$ is a smooth algebraic variety in $\R^n$, then solving $\Var S_\mathrm{CPP}$ reduces to solving a system of polynomial equations. For almost all inputs $a\in \R^n$ this system has a constant number of solutions over the complex numbers, called the \emph{Euclidean distance degree} \cite{DHOST206}. \Cref{main2} thus characterizes the condition number of real solutions of this special system of polynomial equations.
\end{remark}

\section{The condition number of approximation problems} \label{sec_results_1}
\changed{Next, we analyse the condition number of global minimizers of \cref{eqn_IPIP}. This setting will be crucial for a finer interpretation of \cref{main2,main3} in \cref{sec_optimization}.}

\changed{By suitably choosing the height function $\delta$ \cite[Chapter 4, Section 5]{Hirsch1976}, we can assume for each ambient input $a$ in a tubular neighborhood $\Var{T}\subset\R^n$ of $\Var{I}$ that} there is a unique point $x$ on $\Var{I}$ that minimizes the distance from $\Var{I}$ to~$a$. 
% We therefore have the following \changed{\emph{approximation problem} (AP)}
% \begin{align}\label{eqn_APe}
%  \pi_{\Var{O}} \circ \argmin_{(x,y)\in\Var{S}} \frac{1}{2} \| a - \pi_{\Var{I}}(x,y) \|^2
% \end{align}
% \changed{whose} implicit formulation is 
\changed{In this case, the computational problem that consists of finding global minimizers of \cref{eqn_IPIP} can be modeled implicitly with the graph}
\begin{align} \label{eqn_ap} \tag{AP}
 \Var{S}_\mathrm{AP} := \left\{(a,y) \in \Var T \times \Var{O} \mid (\mathrm{P}_{\Var{I}}(a),y)\in\Var{S} \right\},
 \end{align}
 where $\mathrm{P}_{\Var{I}} : \Var T \to \Var{I}, a\mapsto \mathrm{argmin}_{x\in\Var{I}} \Vert x-a\Vert$ is the nonlinear projection that maps $a \in \Var{T}$ to the unique point on $\Var{I}$ that minimizes the distance to $a$. \changed{We refer to the problem $\Var{S}_\mathrm{AP}$ as the \textit{approximation problem} (AP) to distinguish it from the problem of finding critical points of \cref{eqn_IPIP}, which is treated in \cref{sec_results_3}.}

\changed{If $\Var{S}_\mathrm{AP}$ were a smooth submanifold of $\Var T \times \Var{O}$, its condition number would be given by \cref{def_kappa}. However, this is not the case in general. We can nevertheless consider the subset of}
well-posed tuples for \cref{eqn_ap}:
\begin{align*}
 \Var{W}_\mathrm{AP} := \left\{(a,y) \in \Var T \times \Var{O} \mid (\mathrm{P}_{\Var{I}}(a),y)\in\Var{W} \right\}.
\end{align*}
In \cref{sec_proofs_thm2}, we prove the following result under \cref{ass_1}.

\begin{lemma} \label{V_AP_is_a_manifold}
$\Var{W}_\mathrm{AP}$ is a smooth embedded submanifold of $\Var{T} \times \Var{O}$ of dimension~$n$. Moreover, $\Var{W}_\mathrm{AP}$ is dense in $\Var{S}_\mathrm{AP}$.
\end{lemma}

The condition number of \cref{eqn_ap} is now given by \cref{def_kappa} \changed{on $\Var{W}_\mathrm{AP}$. For the remaining points, i.e., $\Var{S}_\mathrm{AP}\setminus\Var{W}_\mathrm{AP}$, the usual definition \cref{computational_problem2} still applies. From this, we obtained the following characterization, which is proved in \cref{sec_proofs_thm2}.}

\begin{theorem}[Global minimizers] \label{main11}
Let \changed{the ambient input} $a \in \Var{T} \subset \R^n$ be sufficiently close to~$\Var{I}$, specifically lying in a tubular neighborhood of $\Var{I}$, so that it has a unique closest point $x \in \Var{I}$. Let $\eta=a-x$. Then,
\[
 \kappa[\Var{S}_\mathrm{AP}](a,y) = \| (\deriv{\pi_{\Var{O}}}{(x,y)})(\deriv{\pi_{\Var{I}}}{(x,y)})^{-1} H_{\eta}^{-1} \|_{\Var{I} \to \Var{O}},
\]
where \changed{$\pi_{\Var{I}}$ and $\pi_{\Var{O}}$ are the projections onto the input and output manifolds respectively, and} $H_{\eta}$ is as in \cref{eqn_Hess_distance}. Moreover, if $(a,y)\in\Var{S}_\mathrm{AP}\setminus\Var{W}_\mathrm{AP}$ then $\deriv{\pi_\Var{I}}{(x,y)}$ is not invertible.
\end{theorem}

\begin{corollary}
$\kappa[\Var{S}_\mathrm{AP}] : \Var{T} \to \R \cup \{+\infty\}$ is continuous \changed{and $\kappa[\Var{S}_\mathrm{AP}](a,y)$ is finite if and only if $(a,y)\in\Var{W}_\mathrm{AP}$}.
\end{corollary}

\changed{\begin{remark} \label{rem_explicit_derivative}
If $f : \Var{I} \to \Var{O}$ is a smooth map and $\Var{S}$ is the graph of~$f$, then the foregoing specializes to
\[
  (\deriv{\pi_{\Var{O}}}{(x,y)})(\deriv{\pi_{\Var{I}}}{(x,y)})^{-1} = \deriv{f}{x}, \quad\text{so}\quad
 \kappa[\Var{S}_\mathrm{AP}](x,y) = \| (\deriv{f}{x}) H_\eta^{-1} \|_{\Var{I}\to\Var{O}}
\]
Similarly, if $F : \Var{O} \to \Var{I}$ is a smooth map with graph $\Var{S}$, then 
\[
 (\deriv{\pi_{\Var{O}}}{(x,y)})(\deriv{\pi_{\Var{I}}}{(x,y)})^{-1} = (\deriv{F}{y})^{-1}, \quad\text{so}\quad
 \kappa[\Var{S}_\mathrm{AP}](x,y) = \| (\deriv{F}{y})^{-1} H_\eta^{-1} \|_{\Var{I}\to\Var{O}}
\]
by the inverse function theorem.
These observations apply to \cref{main3} as well.
\end{remark}}

\changed{The contribution $(\deriv{\pi_\Var{O}}{(x,y)})(\deriv{\pi_\Var{I}}{(x,y)})^{-1}$ to the condition number originates from the solution manifold $\Var{S}$; it is intrinsic and does not depend on the embedding. On the other hand, the curvature factor $H_\eta^{-1}$ is extrinsic in the sense that it depends on the embedding of $\Var{I}$ into $\R^n$, as in \cref{main2}.}

The curvature factor in \cref{main11} disappears when $S_\eta = 0$. This occurs when \changed{the ambient input $a$ lies} on the input manifold, so $\eta=0$. In this case, $H_0^{-1}=\mathbf{1}$, so we have the following result.
\begin{corollary}\label{main1}
Let $(x,y)\in\Var S$. \changed{Then, we have $\kappa[\Var{S}](x,y)=\kappa[\Var{S}_\mathrm{AP}](x,y)$.}
\end{corollary}
% \begin{proof}
% By \cite[Theorem 4]{Rice1966}, $\kappa[\Var{S}](x,y) = \Vert (\deriv{\pi_\Var O}{(x,y)})\,(\deriv{\pi_\Var I}{(x,y)})^{-1}\Vert$. Comparing with \cref{main11} proves the assertion.
% \end{proof}

In other words, if the \changed{ambient input, given by its coordinates in $\R^n$,} lies exactly on $\Var{I}$,
the condition number of \cref{eqn_ap} equals the condition number of the computational problem $\Var{S}$, \changed{even though there can be many more directions of perturbation in $\R^n$ than in $\Var{I}$}!
By continuity, the effect of curvature on the sensitivity of the computational problem can essentially be ignored for small $\Vert \eta\Vert$, i.e., for points very close to~$\Var{I}$. This is convenient in practice because computing the Weingarten map is often more complicated than obtaining the derivatives of the projection maps $\pi_\Var{O}$ and $\pi_\Var{I}$.

\section{The condition number of general critical point problems} \label{sec_results_3}
We now consider the sensitivity of critical points of \cref{eqn_IPIP}. 
As mentioned before, our motivation stems from solving \cref{eqn_IPIP} when a closed-form solution is lacking. Applying Riemannian optimization methods to them, one can in general only guarantee to find points satisfying the first-order optimality conditions \cite{AMS2008}. \changed{Consequently, in theory, they solve} a \emph{generalized critical point problem}: Given an ambient input $a \in \R^n$, produce an output $y \in \Var{O}$ such that there is a corresponding $x \in \Var{I}$ that satisfies the first-order optimality conditions of \cref{eqn_IPIP}.
We can implicitly formulate this computational problem as follows:
\begin{align} \label{eqn_gcpp}\tag{GCPP}
\Var{S}_\mathrm{GCPP} := \left\{(a,x,y) \in \R^n \times\Var{I}\times \Var{O} \mid (a,x)\in\Var{S}_\mathrm{CPP}  \text{ and }  (x,y)\in\Var{S} \right\}.
\end{align}
% which is equally valid if $\Var{S}$ is not a graph of a diffeomorphism.
% Unlike in the input-output framework from \cref{computational_problem},
The graph $\Var{S}_\mathrm{GCPP}$ is defined using three factors, where the first is the input and the third is the output. The second factor $\Var{I}$ encodes how the output is obtained from the input. Triples are ill-posed if either $(x,y)\in\Var{S}$ is ill-posed for the original problem or $(a,x)\in\Var{S}_\mathrm{CPP}$ is ill-posed for the \cref{eqn_cpp}. Therefore, the well-posed locus is
\begin{align*}
\Var{W}_\mathrm{GCPP} := \left\{(a,x,y) \in \R^n \times\Var{I}\times \Var{O} \mid (a,x)\in\Var{W}_\mathrm{CPP}  \text{ and }  (x,y)\in\Var{W} \right\}.
\end{align*}
We prove in \cref{sec_proofs_thm4} that the well-posed tuples form a manifold.

\begin{lemma} \label{V_GCPP_is_a_manifold}
$\Var{W}_\mathrm{GCPP}$ is an embedded submanifold of $\R^n \times \Var{W}$ of dimension $n$.
Moreover, $\Var{W}_\mathrm{GCPP}$ is dense in $\Var{S}_\mathrm{GCPP}$.
\end{lemma}

Consequently, on $\Var{W}_{\mathrm{GCPP}}$ we can use the definition of condition from \cref{def_kappa}. 
Our final main result generalizes \Cref{main11}; it is proved in \cref{sec_proofs_thm4}.

\begin{theorem}[Generalized critical points]\label{main3}
Let $(a,x,y)\in \Var{S}_\mathrm{GCPP}$ and $\eta=a-x$. Then, we have
\[
\kappa[\Var{S}_{\mathrm{GCPP}}](a,x,y) = \Vert (\deriv{\pi_{\Var{O}}}{(x,y)} ) \; (\deriv{\pi_{\Var{I}}}{(x,y)})^{-1} H_{\eta}^{-1}  \Vert_{\Var{I} \to \Var{O}},
\]
where $\pi_{\Var{I}}:\Var{S}\to \Var{I}$ and $\pi_{\Var{O}}:\Var{S}\to \Var{O}$ are coordinate projections, and
$H_\eta$ is given by \cref{eqn_Hess_distance}. Moreover, if $(a,x,y)\in\Var{S}_\mathrm{GCPP}\setminus\Var{W}_\mathrm{GCPP}$, then either $\deriv{\pi_{\Var{I}}}{(x,y)}$ or $H_\eta$ is not invertible.
\end{theorem}

\begin{corollary}\label{cor_kappa_gcpp_cont}
$\kappa[\Var{S}_\mathrm{GCPP}] : \Var{S}_\mathrm{GCPP} \to \R \cup \{+\infty\}$ is continuous \changed{and the condition number $\kappa[\Var{S}_\mathrm{GCPP}](a,x,y)$ is finite if and only if $(a,x,y)\in\Var{W}_\mathrm{GCPP}$}.
\end{corollary}

\changed{Since $\Var{I}$ inherits the Euclidean metric from $\R^n$, the following result is immediately obtained from \cref{main3}.}

\begin{corollary} \label{cor_bounds}
Let $(a,x,y)\in \Var{S}_{\mathrm{GCPP}}$ and $\eta=a-x$. We have
\[
\frac{\kappa[\Var{S}](x,y)}{ \max\limits_{1\leq i\leq m} \big\vert 1-c_i\Vert \eta\Vert\big\vert }
\leq \kappa[\Var{S}_{\mathrm{GCPP}}](a,x,y)
\leq \frac{\kappa[\Var{S}](x,y)}{\min\limits_{1\leq i\leq m} \big\vert 1-c_i\Vert \eta\Vert\big\vert},
\]
where $c_1,\ldots,c_m$ are the principal curvatures of $\Var{I}^m$ at $x$ in direction $\eta=a-x$.
\end{corollary}
\begin{proof}
Since $\kappa[\Var{S}](x,y) = \Vert (\deriv{\pi_\Var O}{(x,y)})\,(\deriv{\pi_\Var I}{(x,y)})^{-1}\Vert_{\Var{I}\to\Var{O}}$, the submultiplicativity of the spectral norm both yields
$\kappa[\Var{S}](x,y)
\leq \kappa[\Var{S}_{\mathrm{GCPP}}](a,x,y) \Vert H_{\eta}\Vert$
and also
$\kappa[\Var{S}_{\mathrm{GCPP}}](a,x,y)
\leq \kappa[\Var{S}](x,y)\Vert H_{\eta}^{-1}\Vert,$
concluding the proof.
\end{proof}

\section{Consequences for Riemannian optimization}\label{sec_optimization}

In \cref{sec_results_2}, we established a connection between the CPP and the Riemannian optimization problem~\cref{eqn_basic_riemannian}: $\Var{S}_\mathrm{CPP}$ is the manifold of \emph{critical tuples} $(a,x) \in \R^n \times \Var{I}$, such that $x$ is a critical point of the squared distance function from $\Var{I}$ to $a$.
For the general implicit formulation in \cref{eqn_gcpp}, however, the connection to Riemannian optimization is \changed{more complicated. For clarity, we therefore also characterize the condition number of local minimizers of \cref{eqn_IPIP} as the condition number of the \textit{unique} global minimizer of a localized Riemannian optimization problem.}

The solutions of the GCPP were defined as the generalization of the AP to critical points. However, we did not show that they too can be seen as the critical tuples of a distance function optimized over in a Riemannian optimization problem. This connection is established next and proved in \cref{sec_proofs_optim}.

\begin{proposition} \label{prop_riemannian_gcpp_cp}
$(a,x,y) \in \Var{W}_\mathrm{GCPP}$ if and only if $(a,(x,y))$ is a well-posed critical tuple of the function optimized over in the Riemannian optimization problem
\[
 \min_{(x,y) \in \Var{W}} \frac{1}{2} \| a - \pi_\Var{I}(x,y) \|^2.
\]
\end{proposition}

Our main motivation for studying critical points instead of only the global minimizer as in \cref{eqn_ap} stems from our desire to predict the sensitivity of outputs of Riemannian optimization methods for solving approximation problems like \cref{eqn_basic_riemannian,eqn_PIP,eqn_IPIP}. When applied to a well-posed optimization problem, these methods in general only guarantee convergence to critical points \cite{AMS2008}. However, in practice, convergence to critical points that are not local minimizers is extremely unlikely \cite{AMS2008}. The reason is that they are \textit{unstable outputs} of these algorithms: Applying a tiny perturbation to such critical points will cause the optimization method to escape their vicinity, and with high probability, converge to a local minimizer instead.

\begin{remark}
One should be careful to separate \emph{condition} from \emph{numerical stability}. The former is a property of a problem, the latter the property of an algorithm.
We do \emph{not} claim that computing critical points of GCPPs other than local minimizers are \emph{ill-conditioned problems}. We only state that many Riemannian optimization methods are \emph{unstable algorithms} for computing them via \cref{eqn_IPIP}. This is not a bad property, since the goal was to find minimizers, not critical points!
\end{remark}

For critical points of the GCPP that are local minimizers \changed{of the problem in \cref{prop_riemannian_gcpp_cp} we can characterize their condition number as the condition number of the unique global minimizer of a localized problem. This result is proved in \cref{sec_proofs_optim}.}

\begin{theorem}[Condition of local minimizers] \label{prop_riemannian_gcpp}
Let the input manifold~$\Var{I}$, the output manifold $\Var O$, and well-posed loci $\Var{W}$ and $\Var{W}_\mathrm{GCPP}$ be as in \cref{eqn_gcpp}. Assume that $(a^\star,x^\star,y^\star) \in \Var{W}_\mathrm{GCPP}$ and $x^\star$ is a local minimizer of $d_{a^\star} :\Var{I} \to \R, x \mapsto \frac{1}{2} \| a^\star - x \|^2$.
Then, there exist open neighborhoods $\Var{A}_{a^\star} \subset \R^n$ around $a^\star$ and  $\Var{N}_{(x^\star,y^\star)} \subset \Var{W}$ around $(x^\star,y^\star)$, so that
\[
 \rho_{(a^\star,x^\star,y^\star)} : \Var{A}_{a^\star} \to \Var{O},\quad a \mapsto \pi_\Var{O} \circ \argmin_{(x,y) \in \Var{N}_{(x^\star,y^\star)}} \frac{1}{2} \| a - \pi_{\Var{I}}(x,y) \|^2
\]
is an \cref{eqn_ap} and the condition number \changed{of this smooth map} satisfies 
\changed{\begin{align*}
\kappa[\rho_{(a^\star,x^\star,y^\star)}](a) 
% &= \kappa[\Var{S}_\mathrm{AP} \cap (\Var{A}_{a^\star} \times \Var{N}_{(x^\star,y^\star)})]\left(a,\rho_{(a^\star,x^\star,y^\star)}(a)\right) \\
&= \kappa[\Var{S}_\mathrm{GCPP}](a, \mathrm{P}_{\Var{I}_{x^\star}}(a), \rho_{(a^\star,x^\star,y^\star)}(a)),
\end{align*}%
where $\Var{I}_{x^\star}=\pi_\Var{I}(\Var{N}_{(x^\star,y^\star)})$ is the projection onto the first factor, and $\mathrm{P}_{\Var{I}_{x^\star}}(a) = \argmin_{x\in\Var{I}_{x^\star}} \| a - x \|^2$ is the nonlinear projection onto $\Var{I}_{x^\star}$.}%
\end{theorem}

Taking $\Var{I}=\Var{O}$ and $\Var{S}$ as the graph of the identity map $\mathbf{1}$, the next result follows.

\begin{corollary}
Let the input manifold $\Var{I}$, and well-posed loci $\Var{W}$ and $\Var{W}_\mathrm{CPP}$ be as in \cref{eqn_cpp}. Assume that $(a^\star,x^\star) \in \Var{W}_\mathrm{CPP}$ is well-posed and that $x^\star$ is a local minimizer of $d_{a^\star}:\Var{I} \to \R,\; x \mapsto \frac{1}{2} \| a^\star - x \|^2$.
Then, there exist open neighborhoods $\Var{A}_{a^\star} \subset \R^n$ around $a^\star$ and $\Var{I}_{x^\star} \subset \Var I$ around $x^\star$, so that %the Riemannian optimization problem
\[
 \rho_{(a^\star,x^\star)}: \Var{A}_{a^\star} \to \Var{I}, \quad a \mapsto \argmin_{x \in \Var{I}_{x^\star}} \frac{1}{2} \| a - x \|^2
\]
is an \cref{eqn_ap} and its condition number is \changed{
\(
\kappa[\rho_{(a^\star,x^\star)}](a) 
= \kappa[\Var{S}_\mathrm{CPP}](a,\rho_{(a^\star,x^\star)}(a)).
\)}
\end{corollary}

\changed{An interpretation of these results is as follows. 
The condition number of a generalized critical point $(a^\star,x^\star,y^\star)$ that happens to correspond to a local minimum of \cref{eqn_IPIP} equals the condition number of the map $\rho_{(a^\star,x^\star,y^\star)}$. This smooth map models a Riemannian AP on the localized computational problem $\Var{N}_{(x^\star,y^\star)} \subset \Var{W}$ which has a unique global minimum for all ambient inputs from $\Var{A}_{a^\star}$. Consequently, every local minimum of \cref{eqn_IPIP} is \textit{robust} in the sense that it moves \textit{smoothly} according to $\rho_{(a^\star,x^\star,y^\star)}$ in neighborhoods of the input $a^\star$ and $(x^\star,y^\star)\in\Var{W}$. In this open neighborhood, the critical point remains a local mimimum.}
% In particular, even though \cref{main3} was only derived to measure the condition number of generalized critical points, it correctly measures the sensitivity of local minimums as well.

\changed{The foregoing entails that if a local minimum $(x^\star,y^\star)\in\Var{W}$ is computed for \cref{eqn_IPIP} with ambient input $a^\star$ using Riemannian optimization, then computing a local minimum $(x, y)$ for \cref{eqn_IPIP} with perturbed input $a^\star + \Delta$ and starting point $(x^\star,y^\star)$ results in $\mathrm{dist}_{\Var{O}}(y,y^\star) \le \kappa[\Var{S}_{\mathrm{GCPP}}](a^\star,x^\star,y^\star) \| \Delta \| + o(\| \Delta \|)$, where $\mathrm{dist}_\Var{O}$ is as in \cref{computational_problem2}, provided that $\|\Delta\|$ is sufficiently small.}

\section{Computing the condition number} \label{sec_computational}
\changed{Three main ingredients are required for evaluating the condition number of \cref{eqn_IPIP} numerically: 
\begin{enumerate}
 \item[(i)] the inverse of the Riemannian Hessian $H_\eta$ from \cref{eqn_Hess_distance},
 \item[(ii)] the derivative of the computational problem $(\deriv{\pi_\Var{O}}{(x,y)})(\deriv{\pi_\Var{I}}{(x,y)})^{-1}$, and 
 \item[(iii)] the spectral norm \cref{eqn_spectral_norm}.
\end{enumerate}
The only additional item for evaluating the condition number of \cref{eqn_IPIP} relative to the idealized inverse problem ``given $x \in \Var{I}$, find a $y\in\Var{O}$ such that $(x,y)\in\Var{S}$'' is the curvature term $H_\eta^{-1}$ from item (i). Therefore, we focus on evaluating (i) and only briefly mention the standard techniques for (ii) and (iii).}

\changed{The derivative in (ii) is often available analytically as in \cref{rem_explicit_derivative}. If no analytical expression can be derived, it can be numerically approximated when $\Var{O}\subset\R^q$ is an embedded submanifold by computing the standard Jacobian of $\Var{S}$, viewed as the graph of a map $\R^n\to\R^q$, and composing it with the projections onto $\Tang{x}{\Var{I}}$ and $\Tang{y}{\Var{O}}$.} 

\changed{Evaluating the spectral norm in (iii) is a linear algebra problem. Let $\Var{I}\subset\R^n$ be embedded with the standard inherited metric and let the Riemannian metric of $\Var{O}$ be given in coordinates by a matrix $G$. If $G = R^T R$ is its Cholesky factorization \cite{matrix_computations}, then \(\| A \|_{\Var{I}\to\Var{O}} = \| R A \|_2 = \sigma_1(RA),\) where $A : \Tang{x}{\Var{I}} \to \Tang{A(x)}{\Var{O}}$ is a linear map and $\|\cdot\|_2$ is the usual spectral norm, i.e., the largest singular value $\sigma_1(\cdot)$. The largest singular value can be computed numerically in $\mathcal{O}( \min\{ m^2 p, p^2 m \} )$ operations, where $m=\dim\Var{I}$ and $p=\dim\Var{O}$, using standard algorithms \cite{matrix_computations}. If the linear map $LA$ can be applied in fewer than $\mathcal{O}(m p)$ operations, a power iteration can be more efficient \cite{TAEP}.}

\changed{The Riemannian Hessian $H_\eta$ can be computed analytically using standard techniques from differential geometry for computing Weingarten maps \cite{Petersen,ONeill1983,ONeill2001,riemannian_geometry,Lee1997}, or it can be approximated numerically \cite{Boumal2015}. Numerical approximation of the Hessian with forward differences is implemented in the Matlab package Manopt \cite{manopt}. Its inverse can in both cases be computed with $\mathcal{O}(m^3)$ operations.}

\changed{We found two main analytical approaches helpful to compute Weingarten maps of embedded submanifolds $\Var{I}\subset\R^n$. The first relies on the following characterization of the Weingarten map $S_{\eta_x}$ in the direction of a normal vector $\eta_x = a - x$. It is defined in \cite[Section~6.2]{riemannian_geometry} as
\begin{equation}\label{def_shape_operator}
S_{\eta_x}: \Tang{x}{\Var{I}} \to \Tang{x}{\Var{I}},\; v_x \mapsto  \mathrm{P}_{\Tang{x}{\Var{I}}}\left( - (\widetilde{\nabla}_{v_x} N)|_x \right),
\end{equation}
where $\mathrm{P}_{\Tang{x}{\Var{I}}}$ denotes the orthogonal projection onto $\Tang{x}{\Var{I}}$, $N$ is a smooth extension of $\eta$ to a normal vector field on an open neighborhood of $x$ in $\Var{I}$, and $\widetilde{\nabla}_{v_x}$ is the Euclidean covariant derivative \cite{Petersen,ONeill1983,ONeill2001,riemannian_geometry,Lee1997}. In practice, the latter can be computed as
\[
(\widetilde{\nabla}_{v_x} Y)|_x = \frac{\mathrm{d}}{\mathrm{d} t} Y|_{\gamma(t)},
\]
where $Y$ is an arbitrary vector field on $\Var{I}$ and $\gamma(t)$ is a smooth curve realizing $v_x$ \cite[Lemma 4.9]{Lee1997}. In other words, $S_{\eta_x}$ maps $\eta_x$ to the tangential part of the usual directional derivative of $N$ in direction of $v_x \in \Tang{x}{\Var{I}} \subset \R^n$.}

% A consequence of \cite[Chapter 6, Proposition 2.1]{riemannian_geometry} is that the definition of \cref{def_shape_operator} is independent of the choice of extension $N$, and hence $S_\eta$ is a well-defined linear map.

\changed{The second approach relies on the \textit{second fundamental form} of $\Var{I}$, which is defined as the projection of the Euclidean covariant derivative onto the normal space $\mathrm{N}_x \Var{I}$:
\[
\SFF_x(X,Y) := ( \SFF(X,Y) )|_x :=  \mathrm{P}_{\NSp{x}{\Var{I}}} \left( (\widetilde{\nabla}_{X|_x} Y)|_x \right),
\]
where now both $X$ and $Y$ are vector fields on $\Var{I}$. The second fundamental form is symmetric in $X$ and $Y$, and $\SFF_x(X,Y)$ only depends on the tangent vectors $X|_x$ and $Y|_x$. It can be viewed as a smooth map $\SFF_x : \Tang{x}{\Var{M}} \times \Tang{x}{\Var{M}} \to \NSp{x}{\Var{M}}$. Contracting the second fundamental form with a normal vector $\eta_x \in\NSp{x}{\Var{I}}$ yields an alternative definition of the Weingarten map: For all $v, w \in \Tang{x}{\Var{I}}$ we have
\(
 \langle S_{\eta_x} (v), w \rangle = \langle \SFF_x(v,w), \eta_x \rangle,
\)
where $\langle\cdot,\cdot\rangle$ is the Euclidean inner product.
Computing the second fundamental form is often facilitated by \emph{pushing forward} vector fields through a diffeomorphism $F : \Var{M} \to \Var{N}$.
In this case, there exists a vector field $Y$ on $\Var{N}$ that is \textit{$F$-related} to a vector field $X$ on $\Var{M}$: $Y|_p = (F_* X)|_p := (\deriv{F}{F^{-1}(p)})(X|_{F^{-1}(p)})$. The integral curves generated by $X$ and $Y = F_* X$ are related by Proposition 9.6 of \cite{Lee2013}. This approach will be illustrated in the next section.}

\changed{As a rule of thumb, we find that computation of the condition number of \cref{eqn_IPIP} is approximately as expensive as one iteration of a Riemannian Newton method for solving this problem. Typically, the cost is not worse than $\mathcal{O}(m^3 + n m^2 + \min\{m^2 p, p^2 m\})$ operations, where the terms correspond one-to-one to items (i)--(iii), and where $\Var{I}^m \subset \R^n$ and $p=\dim\Var{O}$.}

\section{Triangulation in computer vision} \label{sec_triangulation}

A rich source of approximation problems whose condition can be studied with the proposed framework is multiview geometry in computer vision; for an introduction to this domain see \cite{FL2001, HZ2003, Maybank1993}. The tasks consist of recovering information from $r \ge 2$ \textit{camera projections} of a scene in the \textit{world}~$\R^3$. We consider the \emph{pinhole camera model}. In this model the image formation process is modeled by the following transformation, as visualised in \cref{fig_cameras}:
\[
\mu_r: \vect{y} \mapsto \left[\, \frac{A_\ell \vect{y} + \vect{b}_\ell}{\vect{c}_\ell^T \vect{y} + d_\ell}\, \right]_{\ell=1}^r,
\]
where $A_\ell\in\R^{2\times 3}, \vect{b}_\ell\in\R^2, \vect{c}_\ell\in\R^3,$ and $d_\ell\in\R$; see \cite{FL2001, HZ2003, Maybank1993}. Clearly, $\mu_r$ is not defined on all of $\R^3$.  We clarify its domain in \cref{lem_triangulation_diff} below.

The vector $\vect{x} = \mu_r(y) \in \R^{2r}$ obtained by this image formation process is called a \emph{consistent point correspondence}.
Information that can often be identified from consistent point correspondences include camera parameters and scene structure \cite{FL2001, HZ2003, Maybank1993}.

As an example application of our theory we compute the condition number of the \emph{triangulation problem} in computer vision \cite[Chapter 12]{HZ2003}. In this computational problem we have $r \ge 2$ stationary projective cameras and a consistent point correspondence $\vect{x} = (\vect{x}_1, \vect{x}_2, \ldots, \vect{x}_r) \in \R^{2r}$. The goal is to retrieve the world point $\vect y\in\R^3$ from which they originate. Since the imaging process is subject to noisy measurements and the above idealized projective camera model holds only approximately \cite{HZ2003}, we expect that instead of $\vect x$ we are only given $\vect{a} = (\vect{a}_1, \vect{a}_2, \ldots, \vect{a}_r)$ close to $\vect x$. Thus, $\vect{x}$ is the true input of the problem, $\vect{y}$ is the output and $\vect{a}$ is the ambient input.

According to \cite[p.~314]{HZ2003} , the ``gold standard algorithm'' for triangulation solves the (Riemannian) optimization problem
\begin{align}\label{eqn_triangulation}
\min_{\vect{y} \in \R^{3}} \frac{1}{2} \| \vect{a} - \mu_r(\vect{y}) \|^2.
\end{align}
Consequently, the triangulation problem can be cast as a GCCP provided that two conditions hold: (i) the set of consistent point correspondences is an embedded manifold, and (ii) intersecting the back-projected rays is a smooth map from aforementioned manifold to $\R^3$. We verify both conditions in the following subsection.

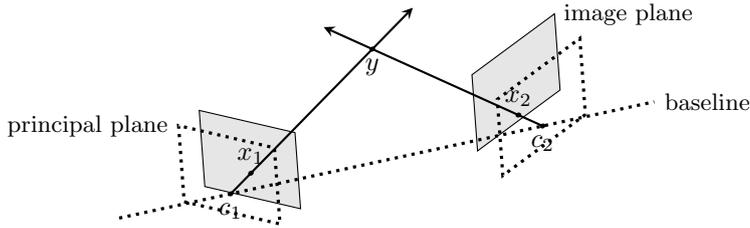
\begin{figure}
\begin{center}
\begin{tikzpicture}[rotate around x=10, rotate around y=25, rotate around z=0, scale=0.55]
\coordinate (c1) at (-3.5,0,1.5);
\coordinate (c2) at ( 3.5,0,1.5);
\coordinate (p) at (0,1,-2);
\coordinate (x1) at (-3,0.1428,1);
\coordinate (x2) at (3,0.1428,1);

\draw[fill=black!10] (-4,-0.5,0) -- (-2,-0.5,2) -- (-2,1.5,2) -- (-4,1.5,0) -- cycle;
\draw[fill=black!10] ( 4,0,0) -- ( 2,0,2) -- ( 2,2,2) -- ( 4,2,0) node[right] {{\small image plane}} -- cycle;
\draw[very thick,dotted] (-4.5,-0.5,0.5) -- (-2.5,-0.5,2.5) -- (-2.5,1.5,2.5) -- (-4.5,1.5,0.5) node[left] {{\small principal plane}} -- cycle;
\draw[very thick,dotted] (4.5,-0.5,0.5) -- (2.5,-0.5,2.5) -- (2.5,1.5,2.5) -- (4.5,1.5,0.5) -- cycle;
\draw[very thick,dotted] (-6,0,1.5) -- (c1) -- (c2) -- (6,0,1.5) node[right] {{\small baseline}};
\draw[thick] (c1) -- (p) edge[-stealth] (1,1.2857,-3);
\draw[thick] (c2) -- (p) edge[-stealth] (-1,1.2857,-3);
\draw[fill] (c1) circle (0.05) node[below] {$\vect{c}_1$};
\draw[fill] (c2) circle (0.05) node[below] {$\vect{c}_2$};
\draw[fill] (p) circle (0.05) node[below] {$\vect{y}$};
\draw[fill] (x1) circle (0.05) node[above] {$\vect{x}_1$};
\draw[fill] (x2) circle (0.05) node[above] {$\vect{x}_2$};
\end{tikzpicture}
\end{center}
\caption{Setup of the $2$-camera triangulation problem. The world coordinates of $\vect{y}\in\R^3$ are to be reconstructed from the projections $\vect{x}_1, \vect{x}_2 \in \R^2$ (in the respective image coordinates) of $\vect{y}$ onto the image planes of the cameras with centers at $\vect{c}_1$ and $\vect{c}_2$ (in world coordinates) respectively.}
\label{fig_cameras}
\end{figure}

\subsection{The multiview manifold}
The image formation process can be interpreted as a projective transformation from the projective space $\Pj^3$ to $\Pj^2$ in terms of the $4 \times 3$ \textit{camera matrices} $P_\ell=\left[\begin{smallmatrix} A_\ell & \vect{b}_\ell \\ \vect{c}_\ell^T & d_\ell\end{smallmatrix}\right]$ \cite{FL2001, HZ2003, Maybank1993}.
It yields homogeneous coordinates of $\vect{x}_\ell = ( z^\ell_{1} / z^\ell_3, z^\ell_2 / z^\ell_3 ) \in \R^2$ where $\vect{z}^\ell = P_\ell \left[\begin{smallmatrix}\vect{y} \\ 1 \end{smallmatrix}\right]$.
Note that if $z_3^\ell = 0$ then the point $\vect{y}$ has no projection\footnote{Actually, it has a projection if we allow points at infinity, i.e., if we consider the triangulation problem in projective space.} onto the image plane of the $\ell$-th camera, which occurs precisely when $\vect{y}$ lies on the \textit{principal plane} of the camera \cite[p.~160]{HZ2003}. This is the plane parallel to the image plane through the camera center $\vect{c}_\ell$; see \cref{fig_cameras}. It is also known that points on the baseline of two cameras, i.e., the line connecting the camera centers visualised by the dashed line in \cref{fig_cameras}, all project to the same two points on the two cameras, called the epipoles \cite[Section 10.1]{HZ2003}. Such points cannot be triangulated from only two images. For simplicity, let $\Var{B} \subset \R^3$ be the union of the principal planes of the first two cameras and their baseline, so $\Var{B}$ is a $2$-dimensional subvariety. The next result shows that a subset of the consistent point correspondences outside of $\Var{B}$ forms a smooth embedded submanifold of $\R^{2r}$.

\begin{lemma}\label{lem_triangulation_diff}
Let $\Var{O}_{\mathrm{MV}} = \R^3 \setminus \Var{B}$ with $\Var{B}$ as above.
The map $\mu_r : \Var{O}_{\mathrm{MV}} \to \R^{2r}$
is a diffeomorphism onto its image $\Var{I}_{\mathrm{MV}} = \mu_r(\Var{O}_{\mathrm{MV}})$.
\end{lemma}
\begin{proof}
Clearly $\mu_r$ is a smooth map between manifolds ($\vect{c}_\ell^T \vect{y} + d_\ell = z_3^\ell \ne 0$). It only remains to show that it has a smooth inverse.
Theorem 4.1 of \cite{HA1997} states that $\mu_2$'s projectivization is a birational map, entailing that $\mu_2$ is a diffeomorphism onto its image. Let $\pi_{1:4} : \R^{2r} \to \R^4$ denote projection onto the first $4$ coordinates. Then, $\mu_2^{-1} \circ \pi_{1:4}$ is a smooth map such that $(\mu_2^{-1} \circ \pi_{1:4}) \circ \mu_r = \mu_2^{-1} \circ \mu_2 = \mathbf{1}_{\Var{O}_{\mathrm{MV}}}$, having used that the domain $\Var{O}_{\mathrm{MV}}$ is the same for all $r$. For the right inverse, we see that any element of $\Var{I}_{\mathrm{MV}}$ can be written as $\mu_r(\vect{y})$ for some $\vect{y} \in \Var{O}_{\mathrm{MV}}$. Hence,
\[
(\mu_r \circ (\mu_2^{-1} \circ \pi_{1:4}))(\mu_r(\vect{y})) = (\mu_r \circ \mathbf{1}_{\Var{O}_{\mathrm{MV}}})(\vect{y}) = \mu_r(\vect{y}),
\]
so it is the identity on $\Var{I}_{\mathrm{MV}}$. This proves that $\mu_r$ has $\mu_2^{-1} \circ \pi_{1:4}$ as smooth inverse. %%\smartqed
\end{proof}

As $\mu_r^{-1}$ is required in the statement of the Weingarten map, we give an algorithm for computing $\mu_r^{-1} = \mu_2^{-1} \circ \pi_{1:4}$.
% The exact triangulation map $\mu_2^{-1}$ is well-known and is stated for example in \cite[Chapter 12]{HZ2003}.
% We obtain the following algorithm as an immediate corollary.
Assume that we are given $\vect{x} \in \Var{I}_\textrm{MV} \subset \R^{2r}$ in the $r$-camera multiview manifold and let its first four coordinates be $(x_1, y_1, x_2, y_2)$. Then, by classic results \cite[Section 12.2]{HZ2003}, the unique element in the kernel of
\[
 \begin{bmatrix}
  (x_1 \vect{e}_3^T - \vect{e}_1^T) P_1 \\
  (y_1 \vect{e}_3^T - \vect{e}_2^T) P_1 \\
  (x_2 \vect{e}_3^T - \vect{e}_1^T) P_2 \\
  (y_2 \vect{e}_3^T - \vect{e}_1^T) P_2 \\
 \end{bmatrix} \in \R^{4 \times 4}
\]
yields homogeneous coordinates of the back-projected point in $\R^3$.

The solution manifold $\Var{S}_\textrm{MV} \subset \Var{I}_\textrm{MV} \times \Var{O}_\mathrm{MV}$ is the graph of $\mu_r^{-1}$. Therefore, it is a properly embedded smooth submanifold; see, e.g., \cite[Proposition 5.7]{Lee2013}.

% \subsection{A local smooth frame}
\subsection{The second fundamental form}
Vector fields and integral curves on $\Var{I}_\textrm{MV}$ can be viewed through the lens of~$\mu_r$:
% For manifolds $\Var{N}$ that can be realized as the diffeomorphic image of $F:\Var{M} \to \Var{N}$, both the vector fields and their integral curves can be studied through the lens of~$F$.
We construct a local smooth frame of $\Var{I}_{\textrm{MV}}$ by pushing forward a frame from $\Var{O}_\textrm{MV}$ by $\mu_r$. Then, the integral curves of each of the local smooth vector fields are computed, after which we apply the Gauss formula for curves \cite[Lemma 8.5]{Lee1997} to compute the second fundamental form.

Because of \cref{lem_triangulation_diff} a local smooth frame for the \textit{multiview manifold} $\Var{I}_{\textrm{MV}} \subset \R^{2r}$ is obtained by pushing forward the constant global smooth orthonormal frame $( \vect{e}_1, \vect{e}_2, \vect{e}_3 )$ of $\R^3$ by the derivative of $\mu_r$.
Its derivative is
\begin{align*}
 \deriv{\mu_r}{\vect{y}} : \Tang{\vect{y}}{\Var{O}_{\textrm{MV}}} \to \Tang{\mu_r(\vect{y})}{\Var{I}_{\textrm{MV}}}, \quad
 \vect{\dot{y}} \mapsto \left[\,\frac{A_\ell \vect{\dot{y}}}{\vect{c}_\ell^T \vect{y} + d_\ell} - (\vect{c}_\ell^T \vect{\dot{y}}) \frac{A_\ell \vect{y} + \vect{b}_\ell}{(\vect{c}_\ell^T \vect{y} + d_\ell)^2}\, \right]_{\ell=1}^r,
\end{align*}
and so a local smooth frame of $\Var{I}_{\mathrm{MV}}$ is given by
\[
 E_i : \Var{I}_{\mathrm{MV}} \to \mathrm{T}\Var{I}_{\mathrm{MV}},\; \mu_r(\vect{y}) \mapsto \left[\, \frac{A_\ell \vect{e}_i}{\vect{c}_\ell^T \vect{y} + d_\ell} - c_{\ell,i} \frac{A_\ell \vect{y} + \vect{b}_\ell}{(\vect{c}_\ell^T \vect{y} + d_\ell)^2} \, \right]_{\ell=1}^r, \quad i=1,2,3,
\]
where $c_{\ell,i}=\vect{c}_\ell^T \vect{e}_i$. It is generally neither orthonormal nor orthogonal.

We compute the second fundamental form of the multiview manifold by differentiation along the integral curves generated by the smooth local frame $(E_1,E_2,E_3)$. The integral curves through $\mu_r(\vect{y})$ generated by this frame are the images of the integral curves passing through $\vect{y}$ generated by the $\vect{e}_i$'s due to \cite[Proposition 9.6]{Lee2013}. The latter are seen to be $g_i(t) = \vect{y} + t \vect{e}_i$ by elementary results on linear differential equations. Therefore, the integral curves generated by $E_i$ are
\[
 \gamma_i(t)
 = \mu_r(g_i(t))
 = \left[\, \frac{A_\ell(\vect{y} + t \vect{e}_i) + \vect{b}_\ell}{\vect{c}_\ell^T (\vect{y} + t \vect{e}_i)  + d_\ell} \, \right]_{\ell=1}^r.
\]
The components of the second fundamental form at $\vect{x} = \mu_r(\vect{y}) \in \Var{I}_{\mathrm{MV}}$ are then
\begin{align*}
\SFF_\vect{x}( E_i, E_j )
 &= \mathrm{P}_{\mathrm{N}\Var{I}_{\mathrm{MV}}} \left( \tfrac{\mathrm{d}}{\mathrm{d}t} E_{j}|_{\gamma_i(t)} \right) \\
&= \mathrm{P}_{\mathrm{N}\Var{I}_{\mathrm{MV}}} \left( \tfrac{\mathrm{d}}{\mathrm{d}t} \begin{bmatrix} \frac{A_\ell \vect{e}_j}{\vect{c}_\ell^T (\vect{y} + t\vect{e}_i) + d_\ell} - {c}_{\ell,j} \frac{A_\ell (\vect{y} + t \vect{e}_i) + \vect{b}_\ell}{(\vect{c}_\ell^T (\vect{y} + t \vect{e}_i) + d_\ell)^2} \end{bmatrix}_{\ell=1}^r
 \right) \\
&= \mathrm{P}_{\mathrm{N}\Var{I}_{\mathrm{MV}}} \left( \begin{bmatrix} - \frac{c_{\ell,i}}{\alpha^2_\ell(\vect{y})} A_\ell \vect{e}_j - \frac{c_{\ell,j}}{\alpha^2_\ell(\vect{y})} A_\ell \vect{e}_i + 2 \frac{c_{\ell,i} c_{\ell,j}}{\alpha_\ell^3(\vect{y})} (A_\ell \vect{y} + \vect{b}_\ell) \end{bmatrix}_{\ell=1}^r  \right),
\end{align*}
where $\alpha_\ell(\vect{y}) = \vect{c}_\ell^T \vect{y} + d_\ell$ and $i,j = 1, 2, 3$.

\subsection{A practical algorithm} \label{sec_MV_practical_algorithm}
The Weingarten map of $\Var{I}_{\mathrm{MV}}$ in the direction of the normal vector $\eta \in \mathrm{N}_{\vect{x}} \Var{I}_{\mathrm{MV}}$ is obtained by contracting the second fundamental form with $\eta$; that is, $\widehat{S}_{\eta} = \langle \SFF_\vect{x}(E_i, E_j), \eta \rangle$. This can be computed efficiently using linear algebra operations. Partitioning $\eta = [ \eta_\ell ]_{\ell=1}^r$ with $\eta_\ell \in \R^2$, the symmetric coefficient matrix of the Weingarten map relative to the frame $\Var{E} = (E_1, E_2, E_3)$ becomes
\[
 \widehat{S}_\eta =
 \sum_{\ell=1}^r \left[\,
 2 \frac{c_{\ell,i} c_{\ell,j}}{\alpha^3_\ell(\vect{y})} \eta_\ell^T (A_\ell \vect{y} + \vect{b}_\ell) - \frac{c_{\ell,i}}{\alpha^2_\ell(\vect{y})} \eta_\ell^T A_\ell \vect{e}_j - \frac{c_{\ell,j}}{\alpha^2_\ell(\vect{y})} \eta_\ell^T A_\ell \vect{e}_i\,
 \right]_{i,j=1}^{3} \in \R^{3 \times 3},
\]
where $\vect{y} = \mu_r^{-1}(\vect{x})$.

To compute the spectrum of the Weingarten map using efficient linear algebra algorithms, we need to express it with respect to an orthonormal local smooth frame by applying Gram--Schmidt orthogonalization to~$\Var{E}$. This is accomplished by placing the tangent vectors $E_1, E_2, E_3$ as columns of a $2r \times 3$ matrix $J$ and computing its compact QR decomposition $Q R = J$. The coefficient matrix of the Weingarten map expressed with respect to the orthogonalized frame $Q=(Q_1, Q_2, Q_3)$ is then $S_\eta = R^{-T} \widehat{S}_\eta R^{-1}$. 
% Indeed, $R^{-T}$ maps the basis $(E_1, E_2, E_3)$ to the orthonormal basis $(Q_1, Q_2, Q_3)$.

Since $\mu_r : \Var{O}_\mathrm{MV} \to \Var{I}_\mathrm{MV}$ is a diffeomorphism, $(\deriv{\pi_{\Var{O}_\mathrm{MV}}}{(\vect{x},\vect{y})}) (\deriv{\pi_{\Var{I}_\mathrm{MV}}}{(\vect{x},\vect{y})})^{-1}$ equals the inverse of the derivative of $\mu_r$ by \cref{rem_explicit_derivative}. As $R$ is the matrix of $\deriv{\mu_r}{\vect{y}} : \Tang{\vect{y}}{\Var{O}_\mathrm{MV}} \to \Tang{\vect{x}}{\Var{I}_\mathrm{MV}}$ expressed with respect to the orthogonal basis $(Q_1,Q_2,Q_3)$ of $\Var{I}_\mathrm{MV}$ and the standard basis of $\R^3$, we get that $\kappa[\Var{S}_{\mathrm{GCPP}}](x+\eta,x,y)$ equals
\begin{align*}
 \| (\deriv{\pi_{\Var{O}_\mathrm{MV}}}{(\vect{x},\vect{y})}) (\deriv{\pi_{\Var{I}_\mathrm{MV}}}{(\vect{x},\vect{y})})^{-1} H_\eta^{-1} \|_{\Var{I}\to\Var{O}}
 = \| R^{-1} (I - S_{\eta})^{-1} \|_2 
 = \frac{1}{\sigma_3\bigl( (I - S_\eta)R \bigr)},
\end{align*}
where $I$ is the $3 \times 3$ identity matrix, and $\sigma_3$ is the third largest singular value.

\changed{The computational complexity of this algorithm grows linearly with the number of cameras $r$. Indeed, the Weingarten map $\widehat{S}_\eta$ can be constructed in $\mathcal{O}(r)$ operations, the smooth frame $\Var{E}$ is orthogonalized in $\mathcal{O}(r)$ operations, the change of basis from $\widehat{S}_\eta$ to $S_\eta$ requires a constant number of operations, and the singular values of a $3\times3$ matrix adds another constant.}

\subsection{Numerical experiments}
The computations below were implemented in Matlab R2017b \cite{matlab}. The code we used is provided as supplementary files accompanying the arXiv version of this article. It uses functionality from Matlab's optimization toolbox. The experiments were performed on a computer running Ubuntu 18.04.3 LTS that consisted of an Intel Core i7-5600U CPU with 8GB main memory.

We present a few numerical examples illustrating the behavior of the condition number of triangulation. The basic setup is described next. We take all $10$ camera matrices $P_i \in \R^{3 \times 4}$ from the ``model house'' data set of the Visual Geometry Group of the University of Oxford \cite{multiviewdata}. These cameras are all pointing roughly in the same direction. In our experiments we reconstruct the point $y := p + 1.5^{10} \vect{v} \in \R^3$, where
$p= (-1.85213, -0.532959, -5.65752)\in\R^3$ is one point of the house and $\vect{v}=(-0.29292,-0.08800, -0.95208)$ is the unit-norm vector that points from the center of the first camera $\vect{p}$.
Given a data point $a \in \R^{2r}$, it is triangulated as follows. First a linear triangulation method is applied, finding the right singular vector corresponding to the least singular value of a matrix whose construction is described (for two points) in \cite[Section 12.2]{HZ2003}. We then use it as a starting point for solving optimization problem \cref{eqn_triangulation} with Matlab's nonlinear least squares solver \texttt{lsqnonlin}. For this solver, the following settings were used: \texttt{TolFun} and \texttt{TolX} set to $10^{-28}$, \texttt{StepTolerance} equal to $10^{-14}$, and both \texttt{MaxFunEvals} and \texttt{MaxIter} set to $10^4$. We provided the Jacobian to the algorithm, which it can evaluate numerically.

\changed{Since the multiview manifold $\Var{I}_\mathrm{MV}$ is only a three-dimensional submanifold of $\R^{2r}$ and the computational complexity is linear in $r$, we do not focus on the computational efficiency of computing $\kappa[\Var{S}_{\mathrm{GCPP}}]$. In practical cases, with $r \ll 1000$, the performance is excellent. For example, the average time to compute the condition number for $r=1000$ cameras was less than $0.1$ seconds in our setup. This includes the time to compute the frame $\Var{E}$, the Weingarten map expressed with respect to $\Var{E}$, the orthogonalization of the frame, and the computation of the smallest singular value.}

\paragraph{Experiment 1}
In the first experiment, we provide numerical evidence for the correctness of our theoretical results. We project $y$ to the image planes of the~$r=10$ cameras: $x = \mu_r(y)$. A random normal direction $\eta \in \NSp{x}{\Var{I}_\mathrm{MV}}$ is sampled by taking a random vector with i.i.d. standard normal entries, projecting it to the normal space, and then scaling it to unit norm. We investigate the sensitivity of the points along the ray $a(t) := x - t \eta$.

The theoretical condition number of the triangulation problem at $a(t)$ can be computed numerically for every $t$ using the algorithm from \cref{sec_MV_practical_algorithm}. We can also estimate the condition number experimentally, for example by generating a large number of small perturbations $a(t) + E(t)$ with $\|E(t)\|\approx0$, solving \cref{eqn_triangulation} numerically using \texttt{lsqnonlin} and then checking the distance to the true critical point $y$. However, from the theory we know that the worst direction of infinitesimal perturbation is given by the left singular vector of $(I - S_{-t\eta})R$ corresponding the smallest (i.e. third) singular value, where $R$ is as in \cref{sec_MV_practical_algorithm}. Let $u(t) \in \R^3$ denote this vector, which contains the coordinates of the perturbation relative to the orthonormal basis $Q=(Q_1,Q_2,Q_3)$; see \cref{sec_MV_practical_algorithm}. Then, it suffices to consider only a small (relative) perturbation in this direction; we took $E(t) := 10^{-6} \|a(t)\| Q u(t)$. We solve \cref{eqn_triangulation} with input $a(t) + E(t)$ and output $y_\text{est}(t)$ using \texttt{lsqnonlin} with the exact solution of the unperturbed problem $y$ as starting point. The experimental estimate of the condition number is then $\kappa_\text{est}(t) = \frac{\|y - y_\text{est}(t) \|}{\| E(t) \|}$.

\begin{figure}\centering
 \includegraphics[height=12em]{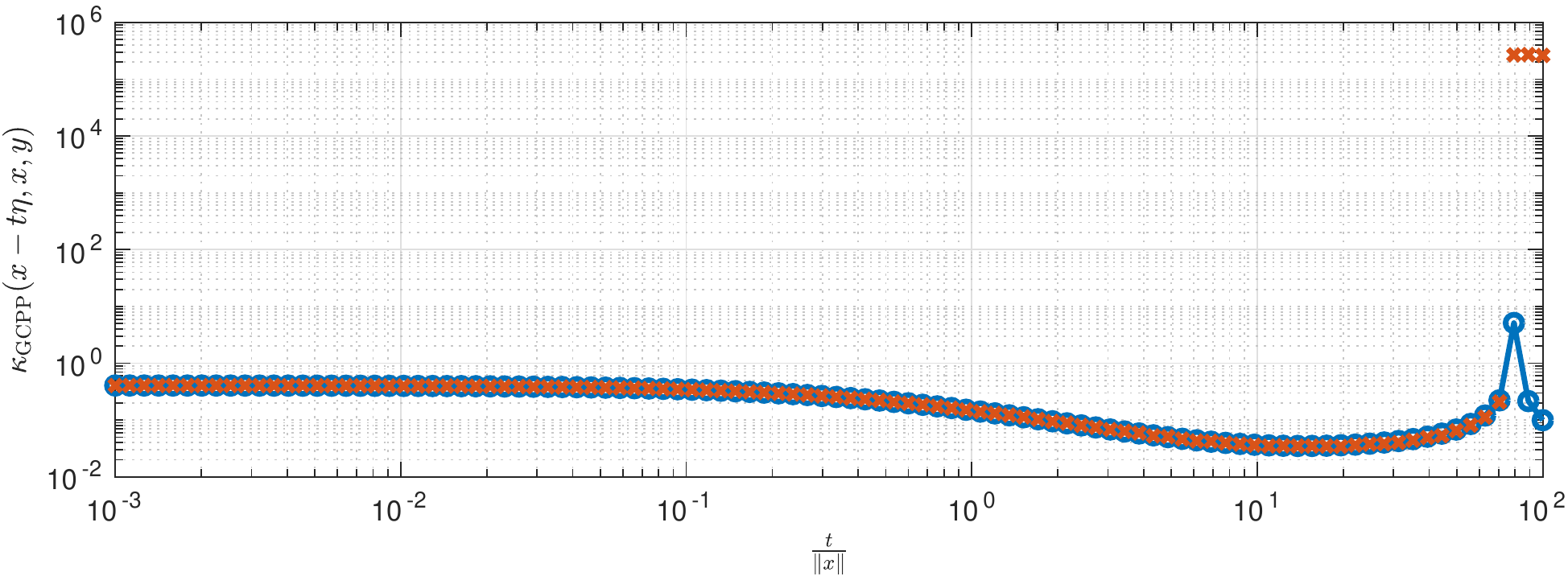}
 \caption{A comparison of theoretical and experimental data of the condition number of the triangulation problem with all cameras for several points along the ray $x - t \eta$. Herein, $\eta$ is a randomly chosen unit-norm normal direction at $x \in \Var{I}_\mathrm{MV}$. The numerically computed theoretical condition numbers are indicated by circles, while the experimentally estimated condition numbers, as described in the text, are plotted with crosses.}
 \label{fig_theory_vs_experiment}
\end{figure}

In \cref{fig_theory_vs_experiment} we show the experimental results for the above setup, where we chose $t = 10^i \|x\|$ with $100$ values of $i$ uniformly spaced between $-3$ and $2$. It is visually evident that the experimental results support the theory. Excluding the three last values of $t$, the arithmetic and geometric means of the condition number divided by the experimental estimate are approximately $1.006664$ and $1.006497$ respectively. This indicates a near perfect match with the theory.

There appears to be one anomaly in \cref{fig_theory_vs_experiment}, however: the three crosses to the right of the singularity at $t^\star = 79.64416 \|x\|$. What happens is that past the singularity, $y$ is no longer a {local minimizer} of the optimization problem $\min_{y\in\Var{O}_\mathrm{MV}} \frac{1}{2} \| a(t) - \mu_{10}(y) \|^2$. This is because $x(t)$ is no longer a local minimizer of $\min_{x \in \Var{I}_\mathrm{MV}} \frac{1}{2} \| a(t) - x \|^2$.
Since we are applying an optimization method, the local minimizer happens to move away from the nearby critical point and instead converges to a different point.
Indeed, one can show that the Riemannian Hessians (see \cite[Chapter 5.5]{AMS2008}) of these problems are positive definite for small $t$, but become indefinite past the singularity $t^*$.
%
% \changed{Indeed, the Riemannian Hessians of these problems are respectively
% \[
%  \widehat{H}_t = J^T J - \widehat{S}_{-t\eta} = R^T R - R^T S_{-t\eta} R = R^T (\mathbf{1} - S_{-t \eta}) R \text{ and } H_t = \mathbf{1} - S_{-t\eta}
% \]
% where $J$ and $R$ are as in \cref{sec_MV_practical_algorithm}. From Sylvester's law of inertia we know that the signatures, i.e., the number of negative, zero, and positive eigenvalues, of $H_t$ and $\widehat{H}_t$ are the same. If $|t|$ is sufficiently small, both are thus positive definite. Specifically, this happens in the tubular neighborhood mentioned in \cref{main11}, where $x(t)$ is a global minimizer of both minimization problems. From \cref{lem_character_1-S} it follows that the condition number has a singularity if and only if $H_t$ is singular. Note that the eigenvalues of $H_t = \mathbf{1} + t S_\eta$ are $1 + t \lambda_i(S_\eta)$, where $\lambda_i(S_\eta)$ is the $i$th eigenvalue of $S_\eta$. So they are linear in $t$. Consequently, as $H_0$ is positive definite and $H_{t^\star}$ is singular, then for all values of $t > t^\star$, $H_t$ (and $\widehat{H}_t$) must be indefinite. In other words, for $t < t^\star$, $x(t)$ is a local minimizer of both problems, but for $t > t^\star$ it is a saddle point. Note that if $S_\eta$ is even negative definite, then $x(t)$ will be a local minimizer for small $t$, a critical point for moderate $t$, and for sufficiently large $t$ it becomes a local maximizer.}

\paragraph{Experiment 2}
The next experiment illustrates how the condition number varies with the distance~$t$ along a normal direction $t \eta \in \mathrm{N}\Var{I}_\mathrm{MV}$. We consider the setup from the previous paragraph. The projection of the point $y$ onto the image planes of the first $k$ cameras is $\vect{x}_k := \mu_k(y)=(P_1 y, \ldots, P_k y)$. A random unit-norm normal vector $\eta_k$ in $\mathrm{N}_{x_k} \Var{I}_\mathrm{MV}$ is chosen as described before. We investigate the sensitivity of the perturbed inputs $a_{k}(t) := x_k + t \eta_k$ for $t \in \R$. The condition number is computed for these inputs for $k=2,3,5,10$ and $\frac{\pm t}{\|x_k\|} = 10^i$ for $10^4$ values of $i$ uniformly spaced between $-3$ and $4$. The results are shown in \cref{fig_pert_normal_direction}. We have $3$ peaks for a fixed $k$ because $\dim \Var{I}_\mathrm{MV} = 3$. The two-camera system has only two singularities in the range $|t| < 10^4 \|x_2\|$. These figures illustrate the continuity of $\kappa_\mathrm{GCPP}$ from \cref{cor_kappa_gcpp_cont}.

Up to about $|t| \le 10^{-2} \|x_k\|$ we observe that curvature hardly affects the condition number of the GCPP for the triangulation problem. Beyond this value, curvature plays the dominant role: both the singularities and the tapering off for high relative errors are caused by it. Moreover, the condition number decreases as the number of cameras increases. This indicates that the use of \emph{minimal problems} \cite{Kukelova2013} in computer vision could potentially introduce numerical instability.

\begin{figure}\centering
 \includegraphics[height=5cm]{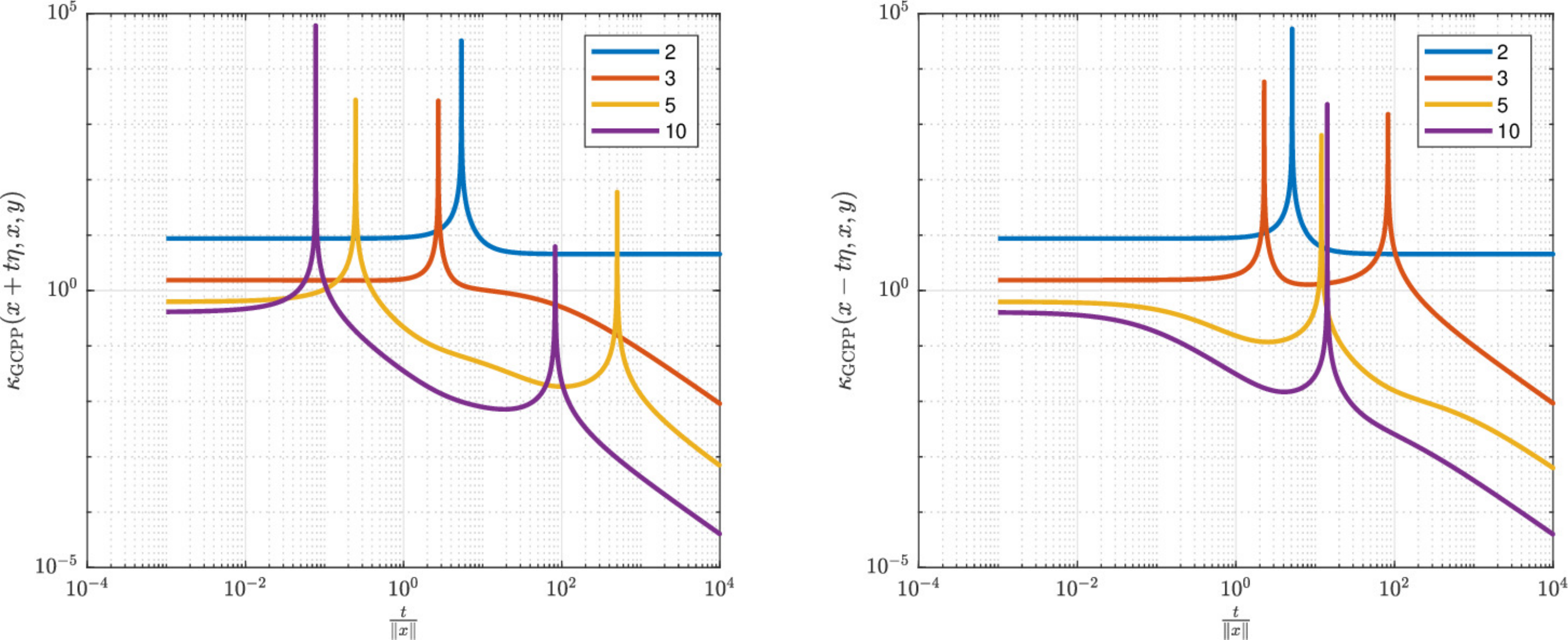}
 \caption{The GCPP condition number in function of the relative distance $\frac{t}{\|x\|}$ from a specified $x \in \Var{I}_\mathrm{MV}$ and a randomly chosen unit-length normal direction $\eta$. Each of the four lines corresponds to a different number $k$ of cameras taking the pictures, namely $k=2,3,5,10$.}
 \label{fig_pert_normal_direction}
\end{figure} 

\section*{Acknowledgements}
The authors thank Nicolas Boumal for posing a question on unconstrained errors, which eventually started the project that led to this article. We are greatly indebted to Carlos Beltr\'an who meticulously studied the manuscript and asked pertinent questions about the theory.
Furthermore, the authors would like to thank Peter B\"urgisser for hosting the second author at TU Berlin in August 2019, and Joeri Van der Veken for fruitful suggestions and pointers to the literature.

\changed{We thank two anonymous reviewers whose remarks enticed us to immensely improve the presentation of our results, specifically in \cref{sec_intro,sec_weingarten}, and to include a discussion of computational aspects in \cref{sec_computational}. We also thank the editor, Anthony So, for adeptly handling the reviewing process.} 

\appendix
\section{Proofs of the technical results} \label{sec_proofs}
\changed{We recall two classic results that we could not locate in modern references in differential geometry.}
A common theme in the three computational problems considered in \cref{sec_results_1,sec_results_2,sec_results_3} is that they involve the projection
\(
\Pi_\Var{I} : \mathrm{N}\Var{I} \to \Var{I}, \;\; (x,\eta) \mapsto x
\)
from the normal bundle $\mathrm{N}\Var{I}$ to the base of the bundle, the embedded submanifold $\Var{I}\subset\R^n$. 
This projection is smooth \cite[Corollary 10.36]{Lee2013} and its derivative is characterized by \changed{the following classical result, which dates back at least to Weyl \cite{Weyl1939}.}

\begin{proposition}\label{thm_ap_derivative0}
  Let $(x,\eta)\in\mathrm{N}\Var I$ and $(\dot x, \dot \eta)\in\mathrm{T}_{(x,\eta)}\mathrm{N}\Var I$. Then,
  $
  \mathrm{P}_{\mathrm{T}_{x}\Var I}(\dot{x}+\dot{\eta}) = H_\eta \dot x,
  $
  where $H_\eta$ is as in \cref{eqn_Hess_distance}.
\end{proposition}

A useful consequence is that it allows us to characterize when $H_\eta$ is singular.

\begin{lemma} \label{lem_character_1-S}
 Let $(a,x) \in \Var{S}_\mathrm{CPP}$, $\eta = a - x$, and $H_\eta$ be as in \cref{eqn_Hess_distance}. Then, $H_\eta$ is invertible if and only if~$(a,x)\in\Var{W}_\mathrm{CPP}$.
\end{lemma}
\begin{proof}
 Let $(\dot{a},\dot{x}) \in \Tang{(a,x)}{\Var{S}_\mathrm{CPP}}$. Then, we have $\dot{x} = (\deriv{\Pi_\Var{I}}{(a,x)})(\dot{a},\dot{x})$ and $\dot{a} = (\deriv{\Pi_{\R^n}}{(a,x)})(\dot{a},\dot{x})$, so that \cref{thm_ap_derivative0} yields the equality of operators
\begin{align}\label{identity_1-S}
\mathrm{P}_{\Tang{x}{\Var{I}}}\, \deriv{\Pi_{\R^n}}{(a,x)} = H_\eta\,\deriv{\Pi_{\Var{I}}}{(a,x)}.
\end{align}
 
On the one hand, if $(a,x) \in \Var{W}_\mathrm{CPP}$ then the map $\deriv{\Pi_{\R^n}}{(a,x)}$ is invertible. By multiplying with its inverse we get $\mathrm{P}_{\Tang{x}{\Var{I}}} = H_\eta (\deriv{\Pi_{\Var{I}}}{(a,x)}) (\deriv{\Pi_{\R^n}}{(a,x)})^{-1}$. As $H_\eta$ is an endomorphism on $\Tang{x}{\Var{I}}$, which is the range of the left-hand side, the equality requires $H_\eta$ to be an automorphism, i.e., it is invertible.

On the other hand, if $(a,x)\not\in \Var{W}_\mathrm{CPP}$, there is some $(\dot{a},\dot{x})\neq 0$ in the kernel of $\deriv{\Pi_{\R^n}}{(a,x)}$. Since $\dot{a} = (\deriv{\Pi_{\R^n}}{(a,x)})(\dot{a},\dot{x}) = 0$, we must have $\dot{x}\neq 0$. Consequently, applying both sides of \cref{identity_1-S} to $(\dot{a},\dot{x})$ gives $0 = H_\eta \dot{x}$, which implies that $H_\eta$ is singular. This finishes the proof. %%\smartqed
\end{proof}

The next result follows almost directly from \cref{thm_ap_derivative0} \changed{and appears several times as an independent statement in the literature}. The earliest reference we could locate is Abatzoglou \cite[Theorem 4.1]{Abatzoglou1978}, who presented it for $C^2$ embedded submanifolds of $\R^n$ in a local coordinate chart. 
% \changed{Other formulations appear in \cite[Lemma 4.1]{DH1994}, \cite[Theorem 3.8]{AM1998}, and \cite[Theorem C]{LS2019}.} The special case of points on the manifold has also been rediscovered several times, for example in \cite[Lemma 4]{AM2012}.

\begin{corollary} \label{thm_ap_derivative}
Let $\Var{I} \subset \R^n$ be a Riemannian embedded submanifold. Then, there exists a tubular neighborhood $\Var{T}$ of $\Var{I}$ such that the projection map $\mathrm{P}_{\Var{I}} : \Var{T} \to \Var{I},  a \mapsto \argmin_{x \in \Var{I}} \| a - x \|^2$ is a smooth submersion with derivative
\[
 \deriv{\mathrm{P}_{\Var{I}}}{a} = H_{\eta}^{-1} \mathrm{P}_{\Tang{x}{\Var{I}}},
\]
where $x = \mathrm{P}_{\Var{I}}(a)$, $\eta = a-x$, and $H_{\eta }$ is given by \cref{eqn_Hess_distance}.
\end{corollary}
% \begin{proof} 
% There exists a tubular neighborhood $\Var{T}$ such that the projection to the base of the bundle is the projection map $\mathrm{P}_\Var{M}$ by \cite[Chapter 4, Section 5]{Hirsch1976}. Then, the first part about the submersion follows from the definition of a tubular neighborhood; see \cite[Proposition 6.25]{Lee2013}.
% 
% The projection $\mathrm{P}_\Var{M}$ takes a point to the closest critical point of the squared distance to the manifold $\Var{M}$. By definition of $\Var{T}$, the closest critical point of $(x,\eta) \in \Var{T} \subset \mathrm{N}\Var{M}$ is the base $x \in \Var{M}$. Therefore, restricted to $\Var{T}$, $\mathrm{P}_\Var{M}|_\Var{T} = \Pi_\Var{M}$, where the latter is the projection to the base of the bundle. Since it is a submersion, $(a,x) \in \Var{T}$ also lives in $\Var{W}_\mathrm{CPP}$.
% The result follows from \cref{lem_character_1-S,identity_1-S}.
% \end{proof}

\changed{In the next subsections, proofs are supplied for the main results.}

\subsection{Proofs for \cref{sec_results_2}} \label{sec_proofs_thm3}

% \begin{proof}[Proof of \cref{C_is_a_manifold}]
% Recall that $\Var{S}_\mathrm{CPP} = \left\{(a,x) \in \mathbb{R}^n \times \Var{I}  \mid  x - a \in \mathrm{N}_{x} \Var{I}\right\}$, and that the normal bundle of $\Var{I}$ is denoted by $\mathrm{N}\Var{I}$.
% Consider the map
% $\psi: \mathrm{N}\Var{I}\to \mathbb{R}^n \times \mathbb{R}^n$ sending $(x,\eta)$ to $(x+\eta,x)$. Since $\psi$ is the restriction of the invertible linear map $\Phi : \R^n \times \R^n \to \R^n \times \R^n, (x,y) \mapsto (x+y,x)$, we see that $\psi$ has maximal rank $n$, so it is a smooth immersion, and that $\psi$ is injective. Since $\mathrm{N}\Var{I}$ is a smooth manifold in the subspace topology of $\R^n \times \R^n$, $\psi$ is also an open map relative to the subspace topologies on its domain and codomain.
% It follows from \cite[Proposition 4.22(a)]{Lee2013} that $\psi$ is a smooth embedding. Hence, the image $\psi(\mathrm{N}\Var{I}) = \Var{S}_\mathrm{CPP}$ is an embedded submanifold by \cite[Proposition 5.2]{Lee2013}.
% \end{proof}

\begin{proof}[Proof of \cref{lem_WCPP_is_manifold}]
Recall that $\Pi_{\R^n}$ is a smooth map, so if its derivative is nonsingular at a point $(a,x)\in\Var{S}_\mathrm{CPP}$, then there exists an open neighborhood of $(a,x)$ where this property remains valid. Therefore, applying the inverse function theorem \cite[Theorem 4.5]{Lee2013} at $(a,x) \in \Var{W}_\mathrm{CPP}$, the fact that $\deriv{\Pi_{\R^n}}{(a,x)}$ is invertible implies that there is a neighborhood $\Var{N}_{(a,x)}$ of $(a,x)$ in $\Var{W}_\mathrm{CPP}$ and a neighborhood of $a$ in $\R^n$ such that $\Var{N}_{(a,x)}$ is diffeomorphic to (an open ball in) $\R^n$. Hence, $\Var{W}_\mathrm{CPP}$ is open in $\Var{S}_\mathrm{CPP}$, and, hence, it is an open submanifold of dimension $n$.

To show that $\Var{W}_\mathrm{CPP}$ is dense in $\Var{S}_\mathrm{CPP}$, take any $(a,x) \in \Var{S}_\mathrm{CPP}\setminus\Var{W}_\mathrm{CPP}$. Let $\eta = a-x$. Then, $(x,x) + (\alpha \eta,0) \in \Var{S}_\mathrm{CPP}$ because $\eta \in \NSp{x}{\Var{I}}$. By \cref{lem_character_1-S}, $(x,x) + (\alpha \eta,0) \in \Var{W}_\mathrm{CPP}$ with $\alpha\in(0,\infty)$ if and only if $H_{\alpha\eta}=\mathbf{1} - S_{\alpha \eta} = \mathbf{1} - \alpha S_{\eta}$ is singular. 
% The latter is singular iff $S_\eta - \frac{1}{\alpha} \mathbf{1}$ is singular. 
This occurs only if $\frac{1}{\alpha}$ equals one of the eigenvalues of the Weingarten map $S_\eta$. Consequently, $(a,x)$ can be reached as the limit of a sequence $(x,x)+ (\alpha_n \eta,0)$, which is wholly contained in $\Var{W}_\mathrm{CPP}$ by taking $\alpha_n$ outside of the discrete set.

Applying \cite[Proposition 5.1]{Lee2013}, we can conclude that $\Var{W}_\mathrm{CPP}$ is even an embedded submanifold. This concludes the argument.
\end{proof}

\begin{proof}[Proof of \cref{main2}]
Let $(a,x) \in \Var{W}_\mathrm{CPP}$ be arbitrary.
On $\Var{W}_\mathrm{CPP}$, the coordinate projection $\Pi_{\R^n}: \Var{S}_\mathrm{CPP} \to \mathbb{R}^n$ has an invertible derivative. Consequently, by the inverse function theorem \cite[Theorem 4.5]{Lee2013} there exist open neighborhoods $\Var X\subset \Var{S}_\mathrm{CPP}$ of $(a,x)$ and $\Var Y$ of $a\subset \R^n$
such that $\Pi_{\R^n}|_{\Var{X}}: \Var X\to \Var Y$ has a smooth inverse function that we call $\phi_{(a,x)}$. Its derivative is $\deriv{\phi_{(a,x)}}{a} = (\deriv{\Pi_{\R^n}}{(a,x)})^{-1}$. Consider the smooth map $\Phi_{(a,x)} := \Pi_{\Var{I}} \circ \phi_{(a,x)}$, where $\Pi_{\Var{I}}: \Var{S}_\mathrm{CPP} \to \Var{I}$ is the coordinate projection. The CPP condition number at $(a,x)\in\Var{W}_\mathrm{CPP}$ \changed{is $\kappa[\Phi_{(a,x)}](a) = \kappa[\Var{W}_{\mathrm{CPP}}](a,x)= \Vert \deriv{\Phi_{(a,x)}}{a}\Vert_{\R^n\to\Var{I}}$.} Since $\deriv{\Phi_{(a,x)}}{a} = (\deriv{\Pi_\Var{I}}{(a,x)}) (\deriv{\Pi_{\R^n}}{(a,x)})^{-1}$, we have that
$\kappa[\Var{W}_{\mathrm{CPP}}](a,x) = \Vert (\deriv{\Pi_{\Var{I}}}{(a,x)})\;(\deriv{\Pi_{\R^n}}{(a,x)})^{-1} \Vert_{\R^n\to\Var{I}}$. Now it follows from \cref{lem_character_1-S,identity_1-S} that $(\deriv{\Pi_{\Var{I}}}{(a,x)}) (\deriv{\Pi_{\R^n}}{(a,x)})^{-1} = H_\eta^{-1} \mathrm{P}_{\Tang{x}{\Var{I}}}$. 
\changed{As the metrics on $\Var{I}$ and $\R^n$ are identical, $\|H_\eta^{-1} \mathrm{P}_{\Tang{x}{\Var{I}}}\|_{\R^n\to\Var{I}} = \|H_\eta^{-1}\|_{\Var{I}\to\Var{I}}$.}
This concludes the first part.

The second part is \cref{lem_character_1-S}.
\end{proof}

\subsection{Proofs for \cref{sec_results_1}} \label{sec_proofs_thm2}

\begin{proof}[Proof of \cref{V_AP_is_a_manifold}]
By \cref{thm_ap_derivative} the projection $\mathrm{P}_{\Var{I}} : \Var{T} \to \Var{I}$ is a smooth map on~$\Var{T}$. Let us denote the graph of this map by $\Var{V} := \{ (a, \mathrm{P}_{\Var{I}}(a) ) \in \Var{T} \times \Var{I} \mid a\in\Var T\}$. It is a smooth embedded submanifold of dimension $\dim \Var{T}= \dim \mathrm{N}\Var I = n$; see, e.g., \cite[Proposition~5.4]{Lee2013}. Let $x=\mathrm{P}_{\Var{I}}(a)$, $\eta = a-x$ and $(\dot{a}, \dot{x}) \in \Tang{(a,x)}{\Var{V}}$ be a nonzero tangent vector.
Because of \cref{thm_ap_derivative} it satisfies $\dot{x} = H_{\eta}^{-1} \mathrm{P}_{\Tang{x}{\Var{I}}} \dot{a}$. Hence,
\(
 (\deriv{\Pi_{\R^n}}{(a,x)})(\dot{a}, H_{\eta}^{-1} \mathrm{P}_{\Tang{x}{\Var{I}}} \dot{a} ) = \dot{a} \ne 0,
\)
so that its kernel is trivial. It follows that $\Var{V} \subset \Var{W}_\mathrm{CPP}$. Since their dimensions match by \cref{C_is_a_manifold}, $\Var{V}$ is an open submanifold. As $\Var{W}_\mathrm{CPP}$ is embedded by \cref{lem_WCPP_is_manifold}, $\Var{V}$ is embedded as well by \cite[Proposition 5.1]{Lee2013}.
Moreover, by construction, we have
 \[
 \Var{W}_\mathrm{AP} = (\Var{V} \times \Var{O}) \cap (\R^n \times \Var{W}) \subset \R^n \times \Var{I} \times \Var{O}.
 \]
The first part of the proof is concluded by observing that the proof of \cref{V_GCPP_is_a_manifold} applies by replacing $\Var{W}_\mathrm{CPP}$ with $\Var{V}$ and $\Var{W}_\mathrm{GCPP}$ with $\Var{W}_\mathrm{AP}$.

Finally, we show that $\Var{W}_\mathrm{AP}$ is dense in $\Var{S}_\mathrm{AP}$. Let $(x+\eta,y)\in\Var{S}_\mathrm{AP}$ with $x\in\Var{I}$ and $\eta\in\mathrm{N}_x \Var{I}$ be arbitrary. Then $\mathrm{P}_\Var{I}(x+\eta) = x$ and $(x,y)\in\Var{S}$.
As $\Var{W}$ is an open dense submanifold of $\Var{S}$, there exists a sequence $(x_i,y_i)\in\Var{W}$ such that $\lim_{i\to\infty} (x_i, y_i) = (x,y)$. Consider $(x_i + \eta_i, y_i) \in \Var{W}_\mathrm{AP}$ with $\eta_i \in \mathrm{N}_{x_i} \Var{I}$ and $\eta_i \to \eta$ chosen in such a way that the first component lies in the tubular neighborhood; this is possible as $\Var{T}$ is an open submanifold of~$\R^n$. Then, the limit of this sequence is $(x+\eta,y)$. Q.E.D.
\end{proof}

\begin{proof}[Proof of \cref{main11}]
For $(a,y) \in \Var{W}_\mathrm{AP}$, let $x=\mathrm{P}_\Var{I}(a)$ and $\eta=a-x$. By \cref{def_kappa}, $\kappa[\Var{W}_\mathrm{AP}](a,y) = \Vert \deriv{(\pi_{\Var{O}} \circ \pi_{\Var{I}}^{-1}\circ \mathrm{P}_\mathcal{I}}{a})\Vert_{\R^n\to\Var{O}}.$
Since $(x,y) \in \Var{W}$, we know that $\deriv{\pi_\Var{I}}{(x,y)}$ is invertible. Moreover, by \cref{thm_ap_derivative} the derivative of the projection $\mathrm{P}_{\Var{I}} : \Var T\to \Var{I}$ at $a$ is $\deriv{\mathrm{P}_{\Var{I}}}{a} = H_{\eta}^{-1} \mathrm{P}_{\Tang{x}{\Var{I}}}$.
Therefore, 
\[
\kappa[\Var{W}_\mathrm{AP}](a,y) = \Vert (\deriv{\pi_{\Var{O}}}{(x,y)}) (\deriv{\pi_{\Var{I}}}{(x,y)} )^{-1}H_{\eta}^{-1} \Vert_{\Var{I}\to\Var{O}}.
\]

For $(a,y) \in \Var{S}_\mathrm{AP}\setminus\Var{W}_\mathrm{AP}$, we have $(x,y)\not\in\Var{W}$ with $x=\mathrm{P}_\Var{I}(a)$. By definition of~$\Var{W}$, this entails that $\deriv{\pi_\Var{I}}{(x,y)}$ is not invertible, concluding the proof.
\end{proof}

% \optional{
% \begin{proof}[Proof of \cref{main1}]
% If $(x,y)\in \Var{S}_\mathrm{AP}\backslash\Var{W}_\mathrm{AP}$, then, by definition, $\kappa_\mathrm{AP}(x,y) = \infty$. On the other hand, if in addition $x\in\Var{I}$, then $(x,y)\not\in\Var{W}$, and so we also have $\kappa(x,y) = \infty$. The two condition numbers are equal in this case. For $(x,y)\in \Var{W}_\mathrm{AP}$, the result follows from \cref{main11} and the fact that $S_0=0$.
% \end{proof}
% }

\subsection{Proofs for \cref{sec_results_3}} \label{sec_proofs_thm4}

\begin{proof}[Proof of \cref{V_GCPP_is_a_manifold}]
Combining \cref{ass_1,lem_WCPP_is_manifold}, it is seen that $\Var{W}_\mathrm{GCPP}$ is realized as the intersection of two smoothly embedded product manifolds:
 \[
  \Var{W}_\mathrm{GCPP} = (\Var{W}_\mathrm{CPP} \times \Var{O}) \cap (\R^n \times \Var{W}) \subset \R^n \times \Var{I} \times \Var{O}.
 \]
Note that $\mathrm{codim}(\Var{W}_\mathrm{CPP} \times \Var{O}) = \dim\Var{I}$, by \cref{C_is_a_manifold}, and $\mathrm{codim}(\R^n \times \Var{W}) = \dim\Var{O}$, because $\dim \Var{W} = \dim \Var{S}$ and \cref{ass_1} stating that $\dim \Var{S} = \dim \Var{I}$. Thus, it suffices to prove that the intersection is \textit{transversal} \cite[Theorem 6.30]{Lee2013}, because then $\Var{W}_\mathrm{GCPP}$ would be an embedded submanifold of codimension $\dim\Var{I} + \dim\Var{O}$.

Consider a point $(a,x,y) \in \Var{W}_\mathrm{GCPP}$. Transversality means that we need to show
\[
\Tang{(a,x,y)}{(\Var{W}_\mathrm{CPP} \times \Var{O})} + \Tang{(a,x,y)}{(\R^n \times \Var{W})}=\Tang{(a,x,y)}{(\R^n \times \Var{I} \times \Var{O})}.
\]
Fix an arbitrary $(\dot{a}, \dot{x}, \dot{y}) \in \Tang{a}{\R^n} \times \Tang{x}{\Var{I}} \times \Tang{y}{\Var{O}}$. Then, it suffices to show there exist $( (\dot{a}_1, \dot{x}_1), \dot{y}_1) \in \Tang{(a,x)}{\Var{W}_\mathrm{CPP}} \times \Tang{y}{\Var{O}}$ and $( \dot{a}_2, (\dot{x}_2, \dot{y}_2))\in \Tang{a}{\R^n} \times \Tang{(x,y)}{\Var{W}}$ such that
\(
 (\dot{a}, \dot{x}, \dot{y}) = (\dot{a}_1 + \dot{a}_2, \dot{x}_1 + \dot{x}_2, \dot{y}_1 + \dot{y}_2).
\)
It follows from \cref{thm_ap_derivative0,lem_character_1-S} that $
 \dot{a}_1 = H_{\eta} \dot{x}_1 + \dot{\eta},
$
where $\eta=a-x$ and where $\dot{\eta} \in \mathrm{N}_x \Var{I}$. The tangent vectors $\dot{y}_1 \in \Tang{y}{\Var{O}}$, $\dot{a}_2 \in \Tang{a}{\R^n}$ and  $\dot{x}_1 \in \Tang{x}{\Var{I}}$ can be chosen freely.
Choosing
$\dot{x}_1 = \dot{x}$ and $\dot{x}_2 = 0$; $\dot{y}_1 = \dot{y}$ and $\dot{y}_2 = 0$ and $\dot{a}_1 =
H_{\eta} \dot{x}_1$ and $\dot{a}_2 = \dot{a} - \dot{a}_1$ yields $(\dot a, \dot x, \dot y)$. This concludes the first part of the proof.

It remains to prove $\Var{W}_\mathrm{GCPP}$ is dense in $\Var{S}_\mathrm{GCPP}$. For this, we recall that
\begin{align*}
% \Var{W}_\mathrm{GCPP} &= (\Var{W}_\mathrm{CPP} \times \Var{O}) \cap (\R^n \times \Var{W})\text{ and }
\Var{S}_\mathrm{GCPP} = (\Var{S}_\mathrm{CPP} \times \Var{O}) \cap (\R^n \times \Var{S}).
\end{align*}
Let $(a,x,y)\in\Var{S}_\mathrm{GCPP}$ and let $\eta = a-x\in\mathrm{N}_x\Var I$. Then, $(a,x) \in \Var{S}_\mathrm{CPP}$ and $(x,y)\in\Var{S}$.
As $\Var{W}$ is an open dense submanifold of $\Var{S}$ by \cref{ass_1}, there exists a sequence $(x_i,y_i)\in\Var{W}$ such that $\lim_{i\to\infty} (x_i, y_i) = (x,y)$. Moreover, there exists a sequence $\eta_i \to \eta$ with $\eta_i \in \Tang{x_i}{\Var{I}}$ so that $(x_i + \eta_i, x_i) \in \Var{S}_\mathrm{CPP}$, because $\Pi_{\Var I}(\Var{S}_\mathrm{CPP}) = \Var I$. Since $\Var{W}_\mathrm{CPP}$ is open dense in $\Var{S}_\mathrm{CPP}$
by \cref{lem_WCPP_is_manifold}, only finitely many elements of the sequence are not in $\Var{W}_\mathrm{CPP}$. We can pass to a subsequence $(x_j+\eta_j,x_j)\in\Var{W}_\mathrm{CPP}$. Now $(x_j + \eta_j, x_j, y_j) \in \Var{W}_\mathrm{GCPP}$ and moreover its limit is $(x+\eta,x,y) \in \Var{S}_\mathrm{GCPP}$. Q.E.D.
\end{proof}

\begin{proof}[Proof of \cref{main3}]
For a well-posed triple $(a,x,y)\in \Var{W}_\mathrm{GCPP}$ we obtain from \cref{def_kappa} that $\kappa[\Var{W}_\mathrm{GCPP}](a,x,y) = \Vert (\deriv{\Pi_{\Var{O}}}{(a,x,y)}) (\deriv{\Pi_{\R^n}}{(a,x,y)})^{-1}\Vert_{\R^n\to\Var{O}}$. We compute the right-hand side of this equation.
The situation appears as follows:
\[
\xymatrix{
\Var{W}_\mathrm{CPP}\ar^{\Pi_{\R^n}}[r]\ar_{\Pi_{\Var{I}}}[d] &\R^n &\Var{W}_\mathrm{GCPP}\ar^{\Pi_{\Var{O}}}[d]\ar_{\Pi_{\R^n}}[l]&\\
\Var{I} & \Var{W}\ar^{\pi_{\Var{I}}}[l]\ar_{\pi_{\Var{O}}}[r] &\Var{O}
}
\]
We have in addition the projections $(\mathbf{1}\times \Pi_{\Var{I}}):\Var{W}_\mathrm{GCPP}\to \Var{W}_\mathrm{CPP},\, (a,x,y)\mapsto (a,x)$ and $(\Pi_{\Var{I}}\times \mathbf{1}):\Var{W}_\mathrm{GCPP}\to \Var{W},\, (a,x,y)\mapsto (x,y)$. With this we have
\[
(\Pi_{\R^n}\circ(\mathbf{1}\times \pi_{\Var{I}}))(a,x,y) = \Pi_{\R^n}(a,x,y) \;\text{ and }\; (\pi_{\Var{O}}\circ(\Pi_{\Var{I}}\times \mathbf{1}))(a,x,y) = \Pi_{\Var{O}}(a,x,y).
\]
Taking derivatives we get
{\small\begin{align*}
(\deriv{\Pi_{\R^n}}{(a,x)})\; \deriv{(\mathbf{1}\times \pi_{\Var{I}})}{(a,x,y)}  = \deriv{\Pi_{\R^n}}{(a,x,y)} \text{ and }
(\deriv{\pi_{\Var{O}}}{(a,x)})\; \deriv{(\Pi_{\Var{I}}\times \mathbf{1})}{(a,x,y)}  = \deriv{\Pi_{\Var{O}}}{(a,x,y)}.\nonumber\end{align*}}%
Since $(x,y)\in \Var{W}$, the derivative $\deriv{ \pi_{\Var{I}}}{(x,y)}$ is invertible, and so $\deriv{(\mathbf{1}\times \pi_{\Var{I}})}{(a,x,y)}$ is also invertible. On other hand, since $(a,x)\in \Var{W}_\mathrm{CPP}$, the derivative $\deriv{ \Pi_{\R^n}}{(a,x)}$ is invertible.
Altogether we see that $\deriv{\Pi_{\R^n}}{(a,x,y)}$ is invertible,
and that
\begin{multline*}
(\deriv{\Pi_{\Var{O}}}{(a,x,y)}) (\deriv{\Pi_{\R^n}}{(a,x,y)})^{-1} \\
   =  (\deriv{\pi_{\Var{O}}}{(x,y)}) \left( \deriv{(\Pi_{\Var{I}}\times \mathbf{1})}{(a,x,y)} \right) \left(\deriv{(\mathbf{1}\times \pi_{\Var{I}})}{(a,x,y)}\right)^{-1}  (\deriv{\Pi_{\R^n}}{(a,x)})^{-1}.
\end{multline*}
The derivatives in the middle satisfy
\begin{align*}
  (\dot x, \dot y) &=\deriv{(\Pi_{\Var{I}}\times \mathbf{1})}{(a,x,y)}\; \left( \deriv{(\mathbf{1}\times \pi_{\Var{I}})}{(a,x,y)} \right)^{-1}\,(\dot a, \dot x),
\end{align*}
while on the other hand 
\(
  (\dot x, \dot y) = (\deriv{\pi_{\Var{I}}}{(x,y)})^{-1}\;(\deriv{\Pi_{\Var{I}}}{(a,x)})\,(\dot a, \dot x)
 .\)
This implies that the two linear maps on the right hand side in the previous equations are equal.
Hence,
\begin{align}
\nonumber (\deriv{\Pi_{\Var{O}}}{(a,x,y)}) (\deriv{\Pi_{\R^n}}{(a,x,y)})^{-1}
&= (\deriv{\pi_{\Var{O}}}{(x,y)}) (\deriv{\pi_{\Var{I}}}{(x,y)})^{-1} (\deriv{\Pi_{\Var{I}}}{(a,x)}) ( \deriv{\Pi_{\R^n}}{(a,x)} )^{-1} \\
\label{222} &= (\deriv{\pi_{\Var{O}}}{(x,y)}) (\deriv{\pi_{\Var{I}}}{(x,y)})^{-1} H_{\eta}^{-1} \mathrm{P}_{\mathrm{T}_x \Var{I}},\end{align}
where we used \cref{identity_1-S,lem_character_1-S}. Taking spectral norms on both sides completes the proof for points in $\Var{W}_\mathrm{GCPP}$.
Finally, for points outside $\Var{W}_\mathrm{GCPP}$ the proof follows from the definition of $\Var{W}_\mathrm{GCPP}$ combined with \cref{main11,main2}.
\end{proof}

\subsection{Proofs for \cref{sec_optimization}}\label{sec_proofs_optim}

% \begin{proof}[Proof of \cref{prop_riemannian_ap}]
% Restricted to the tubular neighborhood $\Var{T}$ where \cref{eqn_ap} is defined, the unique closest point to the manifold is also the closest critical point. The result follows by substituting $x^\star = \mathrm{P}_{\Var{I}}(a^\star)$ and $x = \mathrm{P}_{\Var{I}}(a)$ in \cref{prop_riemannian_gcpp}, and simplifying the resulting statements by eliminating $x$ and $x^\star$.
% \end{proof}

\begin{proof}[Proof of \cref{prop_riemannian_gcpp_cp}]
Taking the derivative of the objective function at a point $(x,y)\in\Var{W}$ we see that the critical points satisfy \(
 \langle a - x, \deriv{\pi_\Var{I}}{(x,y)}(\dot{x}, \dot{y}) \rangle = 0
\)
for all $(\dot{x}, \dot{y}) \in \Tang{(x,y)}{\Var{W}}$. By the definition of $\Var{W}$, the derivative  $\deriv{\pi_\Var{I}}{(x,y)}$ is surjective, which implies that $(\deriv{\pi_\Var{I}}{(x,y)}) (\Tang{(x,y)}{\Var{W}}) = \Tang{x}{\Var{I}}$. Therefore, the condition of being a critical point is equivalent to $a - x\in \mathrm{N}_x\Var I$.
\end{proof}

\begin{proof}[Proof of \cref{prop_riemannian_gcpp}]
By assumption $(a^\star,x^\star,y^\star)\in \Var{W}_\mathrm{GCPP}$. In particular, this implies $(a^\star,x^\star)\in\Var{W}_\mathrm{GCPP}$, so that $\deriv{\Pi_{\R^n}}{(a^\star,x^\star)}:\mathrm{T}_{(a^\star,x^\star)}\Var{W}_\mathrm{CPP}\to \R^n$ is invertible. By the inverse function theorem \cite[Theorem 4.5]{Lee2013}, there is an open neighborhood $\Var{U}_{(a^\star,x^\star)}\subset \Var{W}_\mathrm{CPP}$ of $(a^\star,x^\star)$ such that $\Pi_{\R^n}$ restricts to a diffeomorphism on
\[
\Var{A}_{a^\star}=\Pi_{\R^n}(\Var{U}_{(a^\star,x^\star)}).
\]
Let $\Pi_{\R^n}^{-1}$ denote the inverse of $\Pi_{\R^n}$ on $\Var A_{a^\star}$.
Combining \cref{identity_1-S,lem_character_1-S}, the derivative of $\Pi_{\Var I} \circ \Pi_{\R^n}^{-1}$ is given by
\(
(\deriv{\Pi_{\Var{I}}}{(a,x)}) (\deriv{\Pi_{\R^n}}{(a,x)})^{-1} 
= H_\eta^{-1} \mathrm{P}_{\Tang{x}{\Var{I}}}
\)
and thus has constant rank on  $\Var A_{a^\star}$ equal to $\mathrm{dim}\,\Var{I}$.
% Moreover, the level sets are $(\Pi_{\Var I} \circ \Pi_{\R^n}^{-1})^{-1}(x) = \mathrm{N}_x\Var I$.
This implies that
\[
(\mathbf{1}\times (\deriv{\pi_{\Var{O}}}{(x,y)})\,(\deriv{\pi_{\Var{I}}}{(x,y)})^{-1})(\deriv{\Pi_{\Var{I}}}{(a,x)}) (\deriv{\Pi_{\R^n}}{(a,x)})^{-1}: \mathrm{T}_a \Var A_{a^\star}\to \mathrm{T}_{(x,y)}\Var{W}
\]
has constant rank on $\Var A_{a^\star}$
equal to $\mathrm{dim}\,\Var{I} =\mathrm{dim}\,\Var{W}$, by \cref{ass_1}. Consequently, there is a smooth function
\(
\Phi: \Var A_{a^\star}\to \Var{W}
\)
with $\Phi(a^\star) = (x^\star,y^\star)$. The above showed that its derivative has constant rank equal to $\dim \Var{W}$, so $\Phi$ is a smooth submersion.
By Proposition 4.28 of \cite{Lee2013}, $\Phi(\Var{A}_{a^\star}) \subset \Var{W}$ is an open subset, so it is a submanifold of dimension $\dim \Var{I} = \dim\Var{W}$. Similarly, the derivative $\deriv{\pi_{\Var{I}}}{(x,y)}$ at $(x,y)\in\Var{W}$ has maximal rank equal to $\dim \Var{I}$, by definition of $\Var{W}$. Hence, $\pi_\Var{I} \circ \Phi$ is also a smooth submersion.
After possibly passing to an open subset of $\Var A_{a^\star}$, we can thus assume
\[
\Var N_{(x^\star,y^\star)} :=\Phi(\Var A_{a^\star})\subset \Var{W} \quad\text{and}\quad \Var I_{x^\star} = \pi_{\Var I}(\Var N_{(x^\star,y^\star)}) \subset \Var I
\]
are open submanifolds each of dimension $\mathrm{dim}\,\Var I$ and that
\begin{enumerate}
 \item[(i)] $x^\star$ is the global minimizer of the distance to $a^\star$ on $\Var{I}_{x^\star}$, and hence,
 \item[(ii)] all $x\in\Var{I}_{x^\star}$ lie on the same side of $\Tang{x^\star}{\Var{I}_{x^\star}}$, considering the tangent space $\Tang{x^\star}{\Var{I}_{x^\star}}$ as an affine linear subspace of $\R^n$ with base point $x^\star$.
\end{enumerate}
Now let $a \in \Var{A}_{a^\star}$ be a point in the neighborhood of $a^\star$.
The restricted squared distance function $d_a : \Var N_{(x^\star,y^\star)}\to\R,\, (x,y) \mapsto \frac{1}{2} \| a - \pi_{\Var{I}}(x,y) \|^2$ that is minimized on $\Var N_{(x^\star,y^\star)}$ is smooth.
Its critical points satisfy
\[
\langle a - \pi_{\Var{I}}(x,y), \deriv{\pi_{\Var{I}}}{(x,y)}(\dot{x},\dot{y}) \rangle = 0  \quad\text{for all}\quad (\dot{x},\dot{y}) \in \Tang{(x,y)}{\Var{N}_{(x^\star,y^\star)}}.
\]
Since $(x^\star,y^\star) \in \Var{N}_{(x^\star,y^\star)} \subset \Var{W}$, the derivative $\deriv{\pi_\Var{I}}{(x^\star,y^\star)}$ of $\pi_{\Var{I}} : \Var{W} \to \Var{I}$ is invertible.
Then, for all $(x,y)\in \Var{N}_{(x^\star,y^\star)}$ the image of $\deriv{\pi_{\Var{I}}}{(x,y)}$ is the whole tangent space $\Tang{x}{\Var{I}}$. Consequently, the critical points must satisfy $a - x \perp \Tang{x}{\Var{I}}$.

By construction, for every $a \in \Var{A}_{a^\star}$ there is a unique $(x,y) \in \Var{N}_{(x^\star,y^\star)}$ so that $(a,x,y) \in\Var{W}_{\mathrm{GCPP}}$, namely $(x,y)=\Phi(a)$.
This implies that $(x,y)$ is the \textit{unique} critical point of the squared distance function restricted to $\Var{N}_{(x^\star,y^\star)}$.

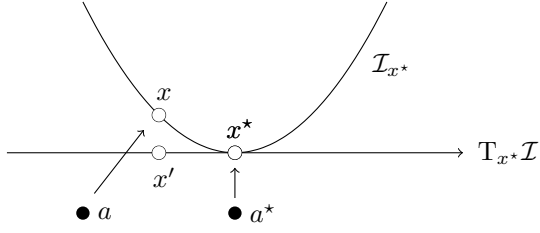
\begin{figure}
  \begin{center}
    \begin{tikzpicture}[scale = 2]

      \draw[black]   plot[smooth,domain=-1:1] (\x, {(\x)*(\x)});

	 \draw[->] (-1.5,0)  -- (1.5,0) node[right, xshift = 2pt] {$\mathrm{T}_{x^\star}\Var I$};
      \draw[black]  (0,0) circle[radius=1.25pt] node[above, yshift=2pt, xshift = 2pt] 	{$x^\star$};
      \fill[white]  (0,0) circle[radius=1.24pt];
      \draw[->] (0,-0.3)  -- (0,-0.1);

	\draw[black]  (0,0) circle[radius=1.25pt] node[above, yshift=2pt, xshift = 2pt] 	{$x^\star$};
      \fill[white]  (0,0) circle[radius=1.24pt];

      \draw[black]  (-0.5,0) circle[radius=1.25pt] node[below, yshift=-2pt, xshift = 2pt] 	{$x'$};
      \fill[white]  (-0.5,0) circle[radius=1.24pt];

       \draw[black]  (-0.5,0.25) circle[radius=1.25pt] node[above, yshift=2pt, xshift = 2pt] 	{$x$};
      \fill[white]  (-0.5,0.25) circle[radius=1.24pt];

      \fill[black]  (0,-0.4) circle[radius=1.25pt] node[right, xshift = 2pt] {$a^\star$};
	\fill[black]  (-1,-0.4) circle[radius=1.25pt] node[right, xshift = 2pt] {$a$};
	 \draw[->] (-0.925,-0.27)  -- (-0.6,0.15);

	 \draw[black]  (1,0.75) node[below, yshift=-2pt, xshift = 2pt] 	{$\Var I_{x^\star}$};
 \end{tikzpicture}
  \end{center}
  \caption{\label{fig10}
  A sketch of the geometric construction in the proof of \cref{prop_riemannian_gcpp}.
  }
\end{figure}

It only remains to show that for all $a\in\Var{A}_{a^\star}$, the minimizer of $d_a$ is attained in the interior of $\Var{N}_{(x^\star,y^\star)}$. This entails that the unique critical point is also the unique global minimizer on $\Var{N}_{(x^\star,y^\star)}$.
We show that there exists a $\delta > 0$ such that restricting $\Var{A}_{a^\star}$ to the open ball $B_\delta(a^\star)$ of radius $\delta$ centered at $a^\star$ in $\R^n$ yields the desired result.

Here is the geometric setup for the rest of the proof, depicted in \cref{fig10}: let $a\in B_\delta(a^\star)$ and let $(x,y)\in \overline{\Var N_{(x^\star,y^\star)}}$ be the minimizer for $d_a$ in the closure of $\Var N_{(x^\star,y^\star)}$. Then, $x$ is in the closure of $\Var I_{x^\star}$ by continuity of $\pi_{\Var I}$. Furthermore, let $x'=P(x)$, where $P := \mathrm{P}_{\Tang{x^\star}{\Var{I}_{x^\star}}}$ is the orthogonal projection onto the tangent space $\Tang{x^\star}{\Var{I}^\star}$ considered as an affine linear subspace of $\R^n$ with base point $x^\star$.

The image of $B_\delta(a^\star)$ under $P$ is an open ball $B_\delta(x^\star) \cap \Tang{x^\star}{\Var{I}_{x^\star}} \subset \Tang{x^\star}{\Var{I}_{x^\star}}$. The distance function $h : \Var{I}_{x^\star} \to \R, \; x \mapsto \|x - P( x ) \|$ measures the ``height'' of $x$ from the tangent space $\Tang{x^\star}{\Var{I}_{x^\star}}$.
Now, we can bound the distance from $a$ to $x$ as follows.
\begin{align}
 \| a - x \|
\nonumber &= \| a - a^\star + a^\star - x^\star + x^\star - x' + x'  - x\| \\
\nonumber &\le \| a - a^\star \| + \| a^\star - x^\star \| + \| x^\star - x' \| + \| x- x' \| \\
\nonumber &\le \| a^\star - x^\star \| + 2 \delta + h(x) \\
\label{eqn_upp_bound} &\le \| a^\star - x^\star \| + 2 \delta + \max_{z \in \Var{P}_{\delta}(x^\star)} h(z),
\end{align}
where the first bound is the triangle inequality, and where $\Var{P}_{\delta}(x^\star) \subset \Var{I}_{x^\star}$ is the preimage of $B_\delta(x^\star)\cap \Tang{x^\star}{\Var{I}_{x^\star}}$ under the projection $P$ restricted to $\Var{I}_{x^\star}$.

Now let $x\in \partial\Var{I}_{x^\star}$ be a point in the boundary of $\Var{I}_{x^\star}$ and, again, $x'=P(x)$.  We want to show that the distance $\Vert a-x\Vert$ is strictly larger than \cref{eqn_upp_bound}. This would imply that local minimizer for $d_a$ is indeed contained in the interior of $\Var N_{(x^\star,y^\star)}$.
We have
\begin{align*}
\| a - x \|
&= \| a - a^\star + a^\star - x^\star + x^\star - x '+ x' - x\| \\
&\ge \| a^\star - x^\star + x^\star - x' + x' - x \| - \| a - a^\star \|.
\end{align*}
Note that $x^\star = P(a^\star)$. Hence we have the orthogonal decomposition
\[
 a^\star - x^\star + x^\star - x' + x' - x = (\mathbf{1} - P)(a^\star - x) + P(a^\star - x) = a^\star - x.
\]
Plugging this into the above, we find
\begin{align*}
\| a - x\|
&\ge \| (\mathbf{1} - P)(a^\star - x) + P(a^\star - x) \| - \delta \\
&= \sqrt{ \|(\mathbf{1} - P)(a^\star - x)\|^2 + \| P(a^\star - x) \|^2 } - \delta \\
&\ge \sqrt{ \|(\mathbf{1} - P)(a^\star - x)\|^2+ \epsilon^2 } - \delta,
\end{align*}
where $\epsilon = \min_{z\in P(\partial\Var{I}_{x^\star})} \Vert x^\star - z\Vert > 0$.

Observe that $a^\star - x^\star$ points from $x^\star$ to $a^\star$ and likewise $x' - x$ points from $x$ to $x'$. Therefore, $\langle a^\star - x^\star, x' - x\rangle \geq 0$ because both point away from the manifold $\Var{I}_{x^\star}$, which lies entirely on the side of $x$ by assumption (ii) above; see also \cref{fig10}. Consequently,
\(
\| (\mathbf{1} - P)(a^\star - x) \| = \Vert a^\star - x^\star +x' - x  \Vert \ge \| a^\star - x^\star \|,
\)
which implies
\begin{align*}
\| a - x \| &\ge \sqrt{ \|a^\star - x^\star\|^2 + \epsilon^2 } - \delta = \|a^\star - x^\star\| + \frac{\epsilon^2}{2 \|a^\star - x^\star\|} - \delta + \mathcal{O}(\epsilon^4).
\end{align*}
Note that $h$ is a continuous function with $\lim_{x\to x^\star} h(x) = 0$. Hence, as $\delta\to0$ the maximum height tends to zero as well. Moreover, $\epsilon$ and $\|a^\star - x^\star\|$ are independent of $\delta$. All this implies that we can choose $\delta >0$ sufficiently small so that
\[
\frac{\epsilon^2}{2 \|a^\star - x^\star\|} - \delta + \mathcal{O}(\epsilon^4)\geq 2 \delta + \max_{z \in \Var{P}_{\delta}(x^\star)} h(z).
\]
It follows that $\Vert a-x\Vert$ is lower bounded by  \cref{eqn_upp_bound}, so that the minimizer for $d_a$ must be contained in the interior of $\Var N_{(x^\star,y^\star)}$.

The foregoing shows that on $B_\delta(a^\star) \cap \Var{A}_{a^\star}$, the map $\Pi_{\Var{I}} \circ \Pi_{\R^n}^{-1}$ equals the nonlinear projection $\mathrm{P}_{\Var{I}_{x^\star}}$ onto the manifold $\Var{I}_{x^\star}$.
Putting everything together, we obtain
\[
 \rho_{(a^\star,x^\star,y^\star)}(a)
 = (\Pi_{\Var{O}} \circ \Pi_{\R^n}^{-1}) (a)
 = (\pi_{\Var{O}} \circ \pi_{\Var{I}}^{-1} \circ \Pi_{\Var{I}} \circ \Pi_{\R^n}^{-1}) (a)
 = (\pi_{\Var{O}} \circ \pi_{\Var{I}}^{-1} \circ \mathrm{P}_{\Var{I}_{x^\star}}) (a);
\]
the second equality following \cref{222}. The condition number of this smooth map is given by~\cref{def_kappa}. Comparing with \cref{main11,main3} concludes the proof.
\end{proof}
 
\bibliographystyle{siam}
\bibliography{BV5}

\begin{thebibliography}{10}

\bibitem{Abatzoglou1978}
{\sc T.~J. Abatzoglou}, {\em The minimum norm projection on {$C^2$} manifolds
  in {$R^n$}}, Trans. Amer. Math. Soc., 243 (1978), pp.~115--122.

\bibitem{AMS2008}
{\sc P.-A. Absil, R.~Mahony, and R.~Sepulchre}, {\em {Optimization Algorithms
  on Matrix Manifolds}}, Princeton University Press, 2008.

\bibitem{AMT2013}
{\sc P.-A. Absil, R.~Mahony, and J.~Trumpf}, {\em An extrinsic look at the
  {Riemannian Hessian}}, in Geometric Science of Information, Lecture Notes in
  Computer Science, Springer, Berlin, Heidelberg, 2013.

\bibitem{AB2012}
{\sc D.~Amelunxen and P.~B\"urgisser}, {\em A coordinate-free condition number
  for convex programming}, SIAM J. Optim., 22 (2012), pp.~1029--1041.

\bibitem{AF2009}
{\sc J.-P. Aubin and H.~Frankowska}, {\em {Set-Valued Analysis}}, Birkh\"auser
  Boston, reprint of the 1990~ed., 2009.

\bibitem{TAEP}
{\sc Z.~Bai, J.~Demmel, J.~Dongarra, A.~Ruhe, and H.~{van der Vorst}}, eds.,
  {\em {Templates for the Solution of Algebraic Eigenvalue Problems: A
  Practical Guide}}, SIAM, Philadelphia, PA, 2000.

\bibitem{BT1997}
{\sc D.~Bau and L.~N. Trefethen}, {\em {Numerical Linear Algebra}}, SIAM, 1997.

\bibitem{BF2009}
{\sc A.~Belloni and R.~M. Freund}, {\em A geometric analysis of {Renegar}'s
  condition number, and its interplay with conic curvature}, Math. Program.,
  Ser. A, 119 (2009), pp.~95--107.

\bibitem{BEB2018}
{\sc T.~Bendory, Y.~C. Eldar, and N.~Boumal}, {\em Non-convex phase retrieval
  from {STFT} measurements}, IEEE Trans. Inform. Theory, 64 (2018),
  pp.~467--484.

\bibitem{BCSS}
{\sc L.~Blum, F.~Cucker, M.~Shub, and S.~Smale}, {\em {Complexity and Real
  Computation}}, Springer-Verlag, New York, 1998.

\bibitem{Boumal2015}
{\sc N.~Boumal}, {\em Riemannian trust regions with finite-difference {Hessian}
  approximations are globally convergent}, in Geometric Science of Information,
  F.~Nielsen and F.~Barbaresco, eds., vol.~2, 2015, pp.~467--475.

\bibitem{Boumal2016}
\leavevmode\vrule height 2pt depth -1.6pt width 23pt, {\em Nonconvex phase
  synchronization}, SIAM J. Optim., 26 (2016), pp.~2355--2377.

\bibitem{Boumal2020}
\leavevmode\vrule height 2pt depth -1.6pt width 23pt, {\em An introduction to
  optimization on smooth manifolds}.
\newblock Available online, 2020.

\bibitem{manopt}
{\sc N.~Boumal, B.~Mishra, P.-A. Absil, and R.~Sepulchre}, {\em {M}anopt, a
  {M}atlab toolbox for optimization on manifolds}, J. Mach. Learn. Res., 15
  (2014), pp.~1455--1459.

\bibitem{BV2017b}
{\sc P.~Breiding and N.~Vannieuwenhoven}, {\em Convergence analysis of
  {Riemannian Gauss--Newton} methods and its connection with the geometric
  condition number}, Appl. Math. Letters, 78 (2018), pp.~42--50.

\bibitem{BV2018}
\leavevmode\vrule height 2pt depth -1.6pt width 23pt, {\em A {Riemannian} trust
  region method for the canonical tensor rank approximation problem}, SIAM J.
  Optim., 28 (2018), pp.~2435--2465.

\bibitem{BC2013}
{\sc P.~B{\"u}rgisser and F.~Cucker}, {\em {Condition: The Geometry of
  Numerical Algorithms}}, Springer, Heidelberg, 2013.

\bibitem{CM2011}
{\sc N.~Chernov and H.~Ma}, {\em Least squares fitting of quadratic curves and
  surfaces}, in Computer Vision, S.~Yoshida, ed., Nova Science Publishers,
  2011.

\bibitem{CC2001}
{\sc D.~Cheung and F.~Cucker}, {\em A new condition number for linear
  programming}, Math. Program., 91 (2001), pp.~163--174.

\bibitem{CCP2008}
{\sc D.~Cheung, F.~Cucker, and J.~Pe\~na}, {\em A condition number for
  multifold conic systems}, SIAM J. Optim., 19 (2008), pp.~261--280.

\bibitem{CC2009}
{\sc A.~Coulibaly and J.-P. Crouzeix}, {\em Condition numbers and error bounds
  in convex programming}, Math. Program., 116 (2009), pp.~79--113.

\bibitem{CP2001}
{\sc F.~Cucker and J.~Pe\~na}, {\em A primal-dual algorithm for solving
  polyhedral conic systems with a finite-precision machine}, SIAM J. Optim., 12
  (2001), pp.~522--554.

\bibitem{dSH2015}
{\sc C.~{Da Silva} and F.~J. Herrmann}, {\em Optimization on the hierarchical
  {Tucker} manifold -- applications to tensor completion}, Linear Algebra
  Appl., 481 (2015), pp.~131--173.

\bibitem{Demmel1987a}
{\sc J.~W. Demmel}, {\em The geometry of ill-conditioning}, J. Complexity, 3
  (1987), pp.~201--229.

\bibitem{Demmel1987}
\leavevmode\vrule height 2pt depth -1.6pt width 23pt, {\em On condition numbers
  and the distance to the nearest ill-posed problem}, Numer. Math., 51 (1987),
  pp.~251--289.

\bibitem{ANLA}
\leavevmode\vrule height 2pt depth -1.6pt width 23pt, {\em {Applied Numerical
  Linear Algebra}}, SIAM, 1997.

\bibitem{riemannian_geometry}
{\sc M.~do~Carmo}, {\em {Riemannian Geometry}}, Birh\"auser, 1993.

\bibitem{DR2009}
{\sc A.~L. Dontchev and R.~T. Rockafeller}, {\em {Implicit Functions and
  Solution Mappings: A View From Varational Analysis}}, Springer Monographs in
  Mathematics, Springer, 2009.

\bibitem{DHOST206}
{\sc J.~Draisma, E.~Horobet, G.~Ottaviani, B.~Sturmfels, and R.~Thomas}, {\em
  The {Euclidean} distance degree of an algebraic variety}, Found. Comput.
  Math., 16 (2016), pp.~99--149.

\bibitem{EF2000}
{\sc M.~Epelman and R.~M. Freund}, {\em Condition number complexity of an
  elementary algorithm for computing a reliable solution of a conic linear
  system}, Math. Program., 88 (2000), pp.~451--485.

\bibitem{EF2002}
\leavevmode\vrule height 2pt depth -1.6pt width 23pt, {\em A new condition
  measure, preconditioners, and relations between different measures of
  conditioning for conic linear systems}, SIAM J. Optim., 12 (2002),
  pp.~627--655.

\bibitem{FL2001}
{\sc O.~Faugeras and Q.~Luong}, {\em {The Geometry of Multiple Images}}, MIT
  Press, Cambridge, MA, 2001.
\newblock The laws that govern the formation of multiple images of a scene and
  some of their applications, With contributions from Th\'{e}o Papadopoulo.

\bibitem{FL2019}
{\sc F.~Feppon and P.~F.~J. Lermusiaux}, {\em The extrinsic geometry of
  dynamical systems tracking nonlinear matrix projections}, SIAM J. Matrix
  Anal. Appl., 40 (2019), pp.~814--844.

\bibitem{FO2005}
{\sc R.~M. Freund and F.~Ord\'o\~nez}, {\em On an extension of condition number
  theory to nonconic convex optimization}, Math. Oper. Res., 30 (2005),
  pp.~173--194.

\bibitem{FV1999}
{\sc R.~M. Freund and J.~R. Vera}, {\em Condition-based complexity of convex
  opitimization in conic linear form via the ellipsoid algorithm}, SIAM J.
  Optim., 10 (1999), pp.~155--176.

\bibitem{matrix_computations}
{\sc G.~H. Golub and C.~{Van Loan}}, {\em {Matrix Computations}}, The John
  Hopkins University Press, 4~ed., 2013.

\bibitem{HZ2003}
{\sc R.~Hartley and A.~Zisserman}, {\em {Multiple View Geometry in Computer
  Vision}}, Cambridge University Press, Cambridge, second~ed., 2003.
\newblock With a foreword by Olivier Faugeras.

\bibitem{HS2018}
{\sc G.~Heidel and V.~Schulz}, {\em A {Riemannian} trust-region method for
  low-rank tensor completion}, Numer. Linear Algebra Appl., 25 (2018).

\bibitem{HA1997}
{\sc A.~Heyden and K.~\AA{}str\"om}, {\em Algebraic properties of multilinear
  constraints}, Math. Meth. Appl. Sci., 20 (1997), pp.~1135--1162.

\bibitem{Higham1996}
{\sc N.~Higham}, {\em {Accuracy and Stability of Numerical Algorithms}}, SIAM,
  Philadelphia, PA, 1996.

\bibitem{Hirsch1976}
{\sc M.~W. Hirsch}, {\em {Differential Topology}}, no.~33 in Graduate Text in
  Mathematics, Springer-Verlag, 1976.

\bibitem{HRS2012}
{\sc S.~Holtz, T.~Rohwedder, and R.~Schneider}, {\em {On manifolds of tensors
  of fixed TT-rank}}, Numer. Math., 120 (2012), pp.~701--731.

\bibitem{HJE2017}
{\sc J.~Humpherys, T.~J. Jarvis, and E.~J. Evans}, {\em {Foundations of Applied
  Mathematics}}, SIAM, 2017.

\bibitem{KSV2014}
{\sc D.~Kressner, M.~Steinlechner, and B.~Vandereycken}, {\em Low-rank tensor
  completion by {Riemannian} optimization}, BIT Numer. Math., 54 (2014),
  pp.~447--468.

\bibitem{Kukelova2013}
{\sc Z.~K\'ukelov\'a}, {\em Algebraic Methods in Computer Vision}, PhD thesis,
  Czech Technical University in Prague, 2013.

\bibitem{Lee1997}
{\sc J.~M. Lee}, {\em {Riemannian Manifolds: Introduction to Curvature}},
  Springer-Verlag, 1997.

\bibitem{Lee2013}
{\sc J.~M. Lee}, {\em {Introduction to Smooth Manifolds}}, Springer, New York,
  USA, 2~ed., 2013.

\bibitem{LW2008}
{\sc Y.~Liu and W.~Wang}, {\em A revisit to least squares orthogonal distance
  fitting of parametric curves and surfaces}, in Advances in Geometric Modeling
  and Processing, F.~Chen and B.~Juttler, eds., Lecture Notes in Computer
  Science, 2008.

\bibitem{matlab}
{\sc MATLAB}, {\em R2017b}.
\newblock Natick, Massachusetts, 2017.

\bibitem{Maybank1993}
{\sc S.~Maybank}, {\em {Theory of Reconstruction from Image Motion}}, vol.~28
  of Springer Series in Information Sciences, Springer-Verlag, Berlin, 1993.

\bibitem{ONeill1983}
{\sc B.~O'Neill}, {\em {Semi-Riemannian Geometry}}, Academic Press, 1983.

\bibitem{ONeill2001}
\leavevmode\vrule height 2pt depth -1.6pt width 23pt, {\em {Elementary
  Differential Geometry}}, Elsevier, revised second edition~ed., 2001.

\bibitem{PR2020}
{\sc J.~Pe\~na and V.~Roshchina}, {\em A data-independent distance to
  infeasibility for linear conic systems}, SIAM J. Optim., 30 (2020),
  pp.~1049--1066.

\bibitem{Petersen}
{\sc P.~Petersen}, {\em {Riemannian Geometry}}, vol.~171 of Graduate Texts in
  Mathematics, Springer, New York, second~ed., 2006.

\bibitem{PORTEOUS71}
{\sc I.~R. Porteous}, {\em The normal singularities of a submanifold}, J.
  Differ. Geom., 5 (1971).

\bibitem{Renegar1995a}
{\sc J.~Renegar}, {\em Incorporating condition measures into the complexity
  theory of linear programming}, SIAM J. Optim., 5 (1995), pp.~506--524.

\bibitem{Renegar1995b}
\leavevmode\vrule height 2pt depth -1.6pt width 23pt, {\em Linear programming,
  complexity theory and elementary functional analysis}, Math. Program., 70
  (1995), pp.~279--351.

\bibitem{Rice1966}
{\sc J.~R. Rice}, {\em A theory of condition}, SIAM J. Numer. Anal., 3 (1966),
  pp.~287--310.

\bibitem{SpivakVol2}
{\sc M.~Spivak}, {\em {A Comprehensive Introduction to Differential Geometry}},
  vol.~2, Publish or Perish, Inc., 1999.

\bibitem{Steinlechner2016}
{\sc M.~Steinlechner}, {\em Riemannian optimization for high-dimensional tensor
  completion}, SIAM J. Sci. Comput., 38 (2016), pp.~S461--S484.

\bibitem{SJ2008}
{\sc A.~Sung-Joon}, {\em Geometric fitting of parametric curves and surfaces},
  J. Inf. Process. Syst., 4 (2008), pp.~153--158.

\bibitem{Thom1962}
{\sc R.~Thom}, {\em Sur la theorie des enveloppes}, J. Math. Pures Appl., 41
  (1962), pp.~177--192.

\bibitem{Vandereycken2013}
{\sc B.~Vandereycken}, {\em Low-rank matrix completion by {R}iemannian
  optimization}, SIAM J. Optim., 23 (2013), pp.~1214--1236.

\bibitem{Vera1996}
{\sc J.~R. Vera}, {\em Ill-posedness and the complexity of deciding existence
  of solutions to linear programs}, SIAM J. Optim., 6 (1996), pp.~549--569.

\bibitem{Vera2014}
\leavevmode\vrule height 2pt depth -1.6pt width 23pt, {\em Geometric measures
  of convex sets and bounds on problem sensitivity and robustness for conic
  linear optimization}, Math. Program., 147 (2014), pp.~47--79.

\bibitem{multiviewdata}
{\sc {Visual Geometry Group, University of Oxford}}, {\em Model house data
  set}, Last accessed 28 september 2020.
\newblock Access online at
  \url{https://www.robots.ox.ac.uk/~vgg/data/data-mview.html}.

\bibitem{Weyl1939}
{\sc H.~Weyl}, {\em On the volume of tubes}, Amer. J. Math., 2 (1939),
  pp.~461--472.

\bibitem{Wilkinson1963}
{\sc J.~H. Wilkinson}, {\em {Rounding Errors in Algebraic Processes}},
  Prentice-Hall, Englewood Cliffs, New Jersey, 1963.

\bibitem{Wilkinson1965}
\leavevmode\vrule height 2pt depth -1.6pt width 23pt, {\em The Algebraic
  Eigenvalue Problem}, Oxford University Press, Amen House, London, United
  Kingdom, 1965.

\bibitem{Wozniakowski1976}
{\sc H.~Wo\'{z}niakowski}, {\em Numerical stability for solving nonlinear
  equations}, Numer. Math., 27 (1976/77), pp.~373--390.

\end{thebibliography}

\end{document}